\documentclass[preprint,12pt]{article}

\RequirePackage[colorlinks,citecolor=blue,urlcolor=blue]{hyperref}
\usepackage{amssymb}
\usepackage{graphicx}
\usepackage{subfigure}
\usepackage{colortbl,dcolumn}
\usepackage{tabularx}
\usepackage{amsmath}
\usepackage{psfrag}
\usepackage{booktabs}
\usepackage{cite}
\usepackage{amsthm}
\usepackage{tikz}
\usetikzlibrary{calc,intersections}
\usetikzlibrary{decorations.pathreplacing}
\newcommand{\bi}{\mathbf{i}}
\newcommand{\bm}{\mathbf{m}}
\newcommand{\bC}{\mathbf{C}}
\newcommand{\bR}{\mathbf{R}}

\newcommand{\N}{\mathbb{N}}
\newcommand{\E}{\mathbb{E}}
\newcommand{\G}{\mathcal{G}}
\newcommand{\F}{\mathcal{F}}
\newcommand{\R}{\mathbb{R}}
\newcommand{\C}{\mathbb{C}}

\newtheorem{tm}{Theorem}[section]

\newtheorem{rk}{Remark}

\newtheorem{ap}{Assumption}

\usepackage[top=1.2in, bottom=1.2in, left=1in, right=1in]{geometry}
\allowdisplaybreaks[4]%change line

\begin{document}

\title{\Large\bf Parareal exponential $\theta$-scheme for longtime simulation of stochastic Schr\"{o}dinger equations with weak damping}
%\author{
%Xu Wang{\footnotemark[2]},\\
% {\small Institute of Computational Mathematics and Scientific/Engineering Computing,}\\{\small Academy of Mathematics and Systems Science, Chinese Academy of Sciences, }\\
%       {\small Beijing 100190, P.R.China }}

       \author{
        {Jialin Hong\footnotemark[1], Xu Wang\footnotemark[2], Liying Zhang\footnotemark[3], }
        }
       \maketitle
       \footnotetext{\footnotemark[1]\footnotemark[2]J. Hong and X. Wang are supported by the National Natural Science Foundation of China (No.91530118, No.91130003, No.11021101, No.91630312 and No.11290142). LSEC, ICMSEC, Academy of Mathematics and Systems Science, Chinese Academy of Sciences, Beijing 100190, P.R.China./
School of Mathematical Sciences, University of Chinese Academy of Sciences, Beijing 100049, P.R.China}
\footnotetext{\footnotemark[3]L. Zhang is supported by National Natural Science Foundation of China (No.11601514). Department of Mathematics, College of Sciences,  China University of Mining and Technology, Beijing 100083, P.R.China.}
        \footnotetext{\footnotemark[2]Corresponding author: wangxu@lsec.cc.ac.cn.}
        
        \begin{abstract}
A parareal algorithm based on an exponential $\theta$-scheme is proposed for the stochastic Schr\"odinger equation with weak damping and additive noise. 
It proceeds as a two-level temporal parallelizable integrator with the exponential $\theta$-scheme as the propagator on the coarse grid. 
The proposed algorithm in the linear case increases the convergence order from one to $k$ for $\theta\in[0,1]\setminus\{\frac12\}$. In particular, the convergence order increases to $2k$ when $\theta=\frac12$ due to the symmetry of the algorithm.  Furthermore, the algorithm is proved to be suitable for longtime simulation based on the analysis of the invariant distributions for the exponential $\theta$-scheme.
The convergence condition for longtime simulation is also established for the proposed algorithm in the nonlinear case, which indicates the superiority of implicit schemes.
Numerical experiments are dedicated to illustrate the best choice of the iteration number $k$, as well as the convergence order of the algorithm for different choices of $\theta$.

\textbf{AMS subject classification:} 
60H35, 65M12, 65W05

\textbf{Key Words: }
stochastic Schr\"odinger equation, parareal algorithm, exponential $\theta$-scheme, invariant measure
\end{abstract}

\section{Introduction}

%{\color{red}(Motivation)}
In the numerical approximation for both deterministic and stochastic evolution equations, several methods have been developed to improve the convergence order of classical schemes, such as (partitioned) Runge-Kutta methods, schemes via modified equations, predictor-corrector schemes and so on (see \cite{BV16,HSW17,ML67,R09} and references therein).
For high order numerical approximations of stochastic partial differential equations (SPDEs), the computing cost can be prohibitively large due to the high dimension in space, especially for longtime simulations. 
It motivates us to study algorithms allowing for parallel implementations to obtain a significant improvement of efficiency.

%{\color{red}(Parareal architectures)}
The parareal algorithm was pioneered in \cite{lions01} as a time discretization of a deterministic partial differential evolution equation on finite time intervals, and was then modified in \cite{MT02} to tackle non-differential evolution equations.
%with pseudo-differential operator $P$:
%\begin{align*}
%\partial_tu+P(u)=0\quad\text{with}\quad P(u)=\widehat{P(\xi)\hat{u}(\xi)},\quad P(\xi)\ge0.
%\end{align*} 
This algorithm is described through a coarse propagator calculated on a coarse grid with step size $\delta T$ and a fine propagator calculated in parallel on each coarse interval with step size $\delta t=\delta T/J$, where $J\in\N_+$ denotes the number of available processors. 
It is pointed out in \cite{lions01} and \cite{MT02} that the error caused by the parareal architecture after a few iterations is comparable to the error caused by a global use of the fine propagator without iteration. 
More specifically, for a fixed iterated step $k\in\N_+$, the parareal algorithm could show order $kp$ with respect to $\delta T$, if a scheme with local truncation error $O(\delta T^{p+1})$ is chosen as the coarse propagator and the exact flow is chosen as the fine propagator. 
Over the past few years, the parareal algorithm has been further studied by \cite{B05,SR05} on its stability, by \cite{gander2014,gander2007} on the potential of longtime simulation, and by \cite{bal2006,E09} on the application to stochastic problems.

%{\color{red}(Stochastic results)} 
When exploring parareal algorithms for stochastic differential equations (SDEs) driven by standard Brownian motions, one of the main differences from the deterministic case is that the stochastic systems are less regular than the deterministic ones. 
Moreover, the convergence order of classical schemes such as explicit Euler scheme, implicit Euler scheme and midpoint scheme, when applied to SDEs, are in general half of those in deterministic case. 
The circumstance becomes even worse when SPDEs are taken into consideration since the temporal regularity of the solution may be worse.
One may not get the optimal convergence rate of the parareal algorithm for the stochastic case following the procedure of the deterministic case.
The author in \cite{bal2006} deals with this problem for SDEs adding assumptions on drift and diffusion coefficients as well as their derivatives, and considers the parareal algorithm when the explicit Euler scheme is chosen as the coarse propagator.
The optimal rate $\frac{k}2(\alpha\wedge3-1)$ is deduced taking advantages of the independency between the increments of Brownian motions,  where $\alpha$ variant for different drift and diffusion coefficients and $\alpha=2$ in general. 

%{\color{red}(Our work: difficulties, procedure, results)}
For the stochastic nonlinear Schr\"odinger equation considered in this paper, there are two main obstacles when establishing implementable parareal algorithms for longtime simulation.
One is that the stiffness caused by the noise 
makes it unavailable to construct parareal algorithms based on existing stable schemes (see e.g. \cite{wang2017}).
It may require higher regularity assumptions due to the iteration adopted in parareal algorithms, see Remark \ref{highregularity}. 
These assumptions are usually not satisfied by SPDEs. 
The other one is that the $\C$-valued nonlinear coefficient does not satisfy one-sided Lipschitz type conditions in general. 
It leads to strict restrictions on the scale of the coarse grid, especially for explicit numerical schemes, when one wishes to get uniform convergence rate.

In this paper, we propose an exponential $\theta$-scheme based parareal algorithms with $\theta\in[0,1]$. 
It allow us to perform the iteration without high regularity assumptions on the numerical solution taking advantages of the semigroup generated by the linear operator of the considered model.
For the linear case with $\theta\in[\frac12,1]$, the exponential $\theta$-scheme possesses a unique invariant Gaussian distribution, which converges to the invariant measure of the exact solution.  
This type of absolute stability ensures the uniform convergence of the proposed parareal algorithm with order $k$ for $\theta>\frac12$ and $2k$ for $\theta=\frac12$. 
If $\theta\in[0,\frac12)$ and the damping $\alpha>0$ is large enough, the uniform convergence still holds. Otherwise, the algorithm is only suitable for simulation over finite time interval, which coincide with the fact that the distribution of the exponential $\theta$ scheme diverges over longtime in this case, see Section \ref{sec3.2}.
For the nonlinear case, we take the proposed algorithm with $\theta=0$ as a keystone to illustrate the convergence analysis for fully discrete schemes with the fine propagator being a numerical solver as well. 
This result is only available over bounded time interval.
To get a time-uniform estimate, internal stage values are utilized in the analysis for the nonlinear case with general $\theta\in[0,1]$. The results give the convergence condition on $\theta$, $L_F$, $\alpha$ and $\delta T$, and indicate that the restriction on $\alpha$ and $\delta T$ is weaker when $\theta$ gets larger.

%{\color{red}(Organization)}
The paper is organized as follows. 
Section \ref{sec2} introduces some notations and assumptions used in the subsequent sections, and gives a brief recall about parareal algorithms. 
Section \ref{sec3} is dedicated to analyze the stability of the parareal exponential $\theta$-scheme by investigating the distribution of the exponential $\theta$-scheme over longtime. 
The rate of convergence for both unbounded and bounded intervals is given for the linear case. 
Section \ref{sec4} focus on the application of the proposed parareal algorithm for the nonlinear case as well as the fully discrete scheme based on the the parareal algorithm. 
Moreover, some modifications are made on the parareal algorithm to release the conditions under which the proposed scheme converges by iteration. 
This improvement is also illustrated through numerical experiments in Section \ref{sec5}.

\section{Preliminaries}
\label{sec2}
We consider the following initial-boundary problem of the stochastic nonlinear Schr\"odinger equation driven by additive noise:
\begin{equation}\label{model}
\begin{aligned}
	&du=\left({\bi}\Delta u-\alpha u+\bi F(u)\right)dt+Q^{\frac12}dW,\\
	&u(t,0)=u(t,1)=0, \quad\quad t\in(0,T],\\
	&u(0,x)=u_0(x), \quad\quad x\in[0,1],
\end{aligned}
\end{equation}
where $\alpha\ge0$ is the damping coefficient and $W(t)$ is an  cylindrical Wiener process defined on the completed filtered probability space $(\Omega,\mathcal{B}, \mathbb{P}, \{\mathcal{B}\}_{t\ge 0})$.  The Karhunen--Lo\`eve expansion of $W$ yields
\begin{align*}
W(t)=\sum_{m=1}^{\infty}e_m(x)\beta_m(t),\quad t\in [0, T],\; x\in[0,1],
\end{align*}
where $\{\beta_m(t)\}_{m\in\N}$ is a family of mutually independent identically distributed $\mathbb{C}$-valued Brownian motions. 

\subsection{Notations}
Throughout this paper, we denote by $H:=L^2(0,1)$ the square integrable space, and denote by $H_0$ the space $H$ with homogenous Dirichlet boundary condition for simplicity.
Then $\{e_m(x)\}_{m\in\N}:=\{\sqrt{2}\sin(m\pi x)\}_{m\in\mathbb{N}}$ is an eigenbasis of the Dirichlet Laplacian in $H$, and the associated eigenvalues of the linear operator $\Lambda:=-\bi\Delta+\alpha$ are expressed as  $\{\lambda_m\}_{m\in\mathbb{N}}:= \left\{{\bf {i}}(m\pi)^2+\alpha\right\}_{m\in\mathbb{N}}$ with $1\le|\lambda_m|\to \infty$ as $m\to\infty$.
Furthermore, we denote the inner product in $H$ by
\begin{align*}
\langle v_1,v_2\rangle:=\int_0^1v_1(x)v_2(x)dx,\quad v_1,v_2\in H.
\end{align*}
In the sequel, we will use the following space
\begin{align*}
\dot{H}^s:=D(\Lambda^{\frac s2})=\left\{u\bigg{|}u=\sum_{m=1}^{\infty}\langle u,e_m\rangle e_m\in H_0,~s.t.,~ \sum_{m=1}^{\infty}|\langle u,e_m\rangle |^2|\lambda_m|^s< \infty
\right\},
\end{align*}
equipped with the norm
\begin{align*}
\|u\|^2_{\dot H^s}=\sum_{m=1}^{\infty}|\langle u,e_m\rangle |^2|\lambda_m|^s,
\end{align*}
which is equivalent to the Sobolev norm $\|\cdot\|_{H^s}$ when $s=0,1,2$.
We use the notation $\|\cdot\|$ instead of $\|\cdot\|_H$ for convenience.

For the nonlinear function $F$ and operator $Q$ in \eqref{model}, we give the following assumptions.
\begin{ap}\label{ap1}
There exists a positive constant $L_F$ such that
\begin{align*}
\|F(v)-F(w)\|\le L_F\|v-w\|,\quad \forall~v,w\in H.
\end{align*}
In addition, $F(0)=0$ and
\begin{align*}
\Im\langle\overline{v},F(v)\rangle=0,\quad\forall~v\in H.
\end{align*}
\end{ap}
\begin{ap}\label{ap2}
Assume that $Q$ is a nonnegative symmetric operator on $H$ with $(-\Delta)^{\frac{s}2}Q^{\frac12}\in\mathcal{L}(H)$ for some $s\ge0$.
\end{ap}

For any $s\ge0,$ the Hilbert--Schmidt norm of operator $Q^{\frac 1 2}$ is defined as 
\begin{align*}
\|Q^{\frac{1}{2}}\|^2_{\mathcal{HS}(H,\dot{H}^s)}
:=\sum_{m=1}^{\infty}\|Q^{\frac{1}{2}}e_m\|_{\dot H^s}^2
=\|(-\Delta^{\frac{s}2})Q^{\frac{1}{2}}\|^2_{\mathcal{HS}(H,H)}.
\end{align*}

Let $S(t):= e^{-t\Lambda}$ be the semigroup generated by  operator $\Lambda$. The mild solution of (\ref{model}) exists globally under Assumptions \ref{ap1} and \ref{ap2} with the following form
\begin{align}\label{mild}
u(t)=S(t)u_0+\bi \int_{0}^{t}S(t-s)F(u)ds+\int_{0}^{t}S(t-s)Q^{\frac 1 2}dW(s).
\end{align}
For any $0\le r\le l$, it holds 
\begin{align*}
\|S(t)\|_{\mathcal{L}(\dot{H}^l,\dot{H}^r)}
:=\sup_{v\in\dot H^s}\frac{\|S(t)v\|_{\dot H^r}}{\|v\|_{\dot H^l}}
\le e^{-\alpha t}.
\end{align*}

\subsection{Framework of parallelization in time}
%This section is devoted to give the parareal algorithm based on a fine propagator $\F$ and a coarse propagator $\G$. Namely, the coarse propagator $\G$ is chosen as the exponential Euler scheme in the present paper. The stability and convergence accuracy of the parareal algorithm is discussed when $\F$ denotes the exact solution.
%\subsection{Algorithm}
In this section, we briefly recall the procedure of parareal algorithms, which are constructed through the interaction of a coarse and a fine propagators under different time scales. The parareal algorithm, or equivalently the time-parallel algorithm, consists of four parts in general: interval partition, initialization, time-parallel computation, and correction. The numerical solution is expected to converge fast by iteration to the solution of a global use of fine propagator $\F$. 
\subsubsection{Interval partition}
The considered interval $[0,T]$ is first divided into $N$ parts with a uniform coarse step size $\delta T=T_{n}-T_{n-1}$ for any $n=1,\cdots,N$ as follows.

\vspace{1em}
\begin{tikzpicture}[xscale = 6.5]
\centering{
\draw [white] (0.15,0) -- (0.3,0);
\draw [very thick] (0.3,0) -- (1.8,0);
\node[align = center, above] at (1.05, 0.3)
        {$\delta T$};
\draw [thick] (0.3, 0) node [below] {$T_0=0$} to (0.3, 0.2);
\draw [thick] (0.6, 0) node [below] {} to (0.6, 0.2);
\draw [thick] (0.9, 0) node [below] {$T_{n-1}$} to (0.9, 0.2);
\draw [thick] (1.2, 0) node [below] {$T_{n}$} to (1.2, 0.2);
\draw [thick] (1.5, 0) node [below] {} to (1.5, 0.2);
\draw [thick] (1.8, 0) node [below] {$T_{N}=T$} to (1.8, 0.2);
\draw[thick,decorate,decoration={brace,amplitude=3pt},xshift=0pt,yshift=1.2pt]
(0.9,0.2) -- (1.2,0.2) node [black,midway,xshift=15pt]{};
}
\end{tikzpicture}

Each subinterval is further divided into $J$ parts with a uniform fine step size $\delta t=t_{n,j+1}-t_{n,j}=\frac{\delta T}{J}$ for any $n=0,\cdots,N-1$ and $j=1,\cdots,J-1$.
It satisfies that $t_{n-1,0}=T_{n-1}$ and $t_{n-1,J}=:t_{n,0}$.

\vspace{1em}
\begin{tikzpicture}[xscale = 6]
%\draw [white] (-0.2,0) -- (0.3,0);
\draw [very thick] (0.3,0) -- (1.8,0);
\node[align = center, above] at (1.05, 0.3)
        {$\delta t$};
\draw [thick] (0.3, 0) node [below] {$t_{n-1,0}=T_{n-1}$} to (0.3, 0.2);
\draw [thick] (0.6, 0) node [below] {} to (0.6, 0.2);
\draw [thick] (0.9, 0) node [below] {$t_{n-1,j}$} to (0.9, 0.2);
\draw [thick] (1.2, 0) node [below] {$t_{n-1,j+1}$} to (1.2, 0.2);
\draw [thick] (1.5, 0) node [below] {} to (1.5, 0.2);
\draw [thick] (1.8, 0) node [below] {$t_{n-1,J}=T_{n}$} to (1.8, 0.2);
\draw[thick,decorate,decoration={brace,amplitude=3pt},xshift=0pt,yshift=1.2pt]
(0.9,0.2) -- (1.2,0.2) node [black,midway,xshift=15pt]{};
\end{tikzpicture}

If the value at the coarse grid $\{T_n\}_{n=0}^N$ is given, denoted by $\{u_n\}_{n=0}^N$, the numerical solutions at the fine grid $\{t_{n-1,j}\}_{j=1}^J$ on each subinterval $[T_{n-1},T_{n}]$ can be calculated independently by choosing $u_{n-1}$ as the initial value over the subinterval.

\subsubsection{Initialization}
We define a coarse propagator $\G$
\begin{align}\label{G}
u_{n}=\G(T_{n},T_{n-1},u_{n-1})
\end{align}
based on some specific scheme to gain a numerical solution $\{u_n\}_{n=0}^N$ at coarse grid $\{T_n\}_{n=0}^N$.

The coarse propagator $\G$ gives a rough approximation on the coarse grid $\{T_n\}_{n=0}^N$, which makes it possible to calculate the numerical solutions on each subinterval parallel to one another. 
In general, $\G$ is required to be easy to calculate and need not to be of high accuracy. On the other hand, the fine propagator $\F$ defined on each subinterval is assumed to be more accurate than $\G$ to ensure that the proposed parareal algorithm is accurate enough.

\subsubsection{Time-parallel computation}
%Based on the interval partition and initialization, we now can calculate the numerical solutions on each subinterval parallel to one another. 
We consider the subinterval $[T_{n-1},T_{n}]$ with initial value $u_{n-1}$ at $T_{n-1}$, and apply a fine propagator $\F$ over this subinterval.  
More precisely, we denote by $\hat u_{n-1,1}:=\F(t_{n-1,1},t_{n-1,0},u_{n-1})$ the one step approximation obtained by $\F$ starting from $u_{n-1,0}:=u_{n-1}$ at time $t_{n-1,0}:=T_{n-1}$, see Figure \ref{fig1}.
Thus, the numerical solution at time $t_{n,j}$ can be expressed as
\begin{align*}
\hat u_{n-1,j}=\F(t_{n-1,j},t_{n-1,j-1},\hat u_{n-1,j-1})=\F(t_{n-1,j},t_{n-1,0},\hat u_{n-1,0}),\quad \forall~j=1,\cdots,J.
\end{align*}
For $j=J$, we get $\hat u_{n-1,J}=\F(T_{n},T_{n-1},u_{n-1})$ which is $\mathcal{B}_{T_{n}}$-adapted.

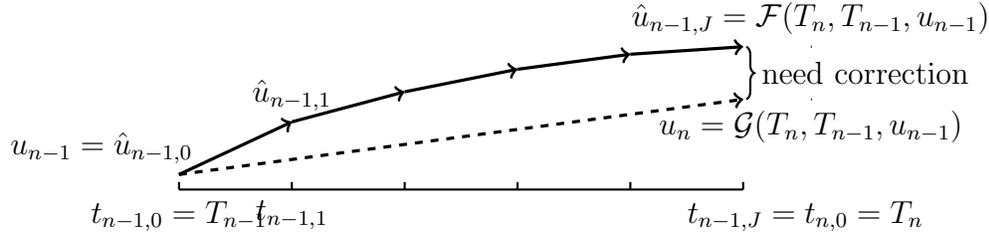
\begin{figure}[h]
\begin{tikzpicture}[xscale = 1.5]
%\draw [white] (-1.5,0) -- (0,0);
\draw[thick] (0,-0.2) to  (5,-0.2);
\draw[->, dashed,very thick] (0,0) to  (5,1);
\draw[->, very thick] (0,0) to  (1,0.7);
\draw[->, very thick] (1,0.7) to  (2,1.1);
\draw[->, very thick] (2,1.1) to (3,1.4);
\draw[->, very thick] (3,1.4) to (4,1.6);
\draw[->, very thick] (4,1.6) to (5,1.7);
\draw [thick] (0, -0.2) node [below] {$t_{n-1,0}=T_{n-1}$} to (0, -0.1);
\draw [thick] (5, -0.2) node [below] {$\quad\quad\quad\quad t_{n-1,J}=t_{n,0}=T_{n}$} to (5, -0.1);
\draw [thick] (1, -0.2) node [below] {$t_{n-1,1}$} to (1, -0.1);
\draw [thick] (2, -0.2) node [below] {} to (2, -0.1);
\draw [thick] (3, -0.2) node [below] {} to (3, -0.1);
\draw [thick] (4, -0.2) node [below] {} to (4, -0.1);

\draw [thick] (0, 0) node [above] {$u_{n-1}=\hat{u}_{n-1,0}~~~~~~~~~~~~~~~$} to (0, 0);
\draw [thick] (5.6, 1) node [below] {$u_{n}=\G(T_{n},T_{n-1},u_{n-1})$} to (5.6, 1);
\draw [thick] (1, 0.7) node [above] {$\hat{u}_{n-1,1}$} to (1,0.7);
\draw [thick] (5.6, 1.7) node [above] {$\hat{u}_{n-1,J}=\F(T_{n},T_{n-1},u_{n-1})$} to (5.6,1.7);
\draw[thick,decorate,decoration={brace,amplitude=3pt,mirror},xshift=0.8pt,yshift=0pt]
(5,1) -- (5,1.7) node [black,midway,xshift=45pt]{need correction};
\end{tikzpicture}
\caption{Numerical solutions obtained by $\F$ and $\G$ on $[T_{n-1},T_{n}]$}
\label{fig1}
\end{figure}

\subsubsection{Correction}
Note that we get two numerical solutions $u_{n}$ and $\hat u_{n-1,J}$ at time $T_{n}$ from above procedure, which are not equal to each other in general, see Figure \ref{fig1}. Some correction should be applied to get a family of numerical solution on the grid $\{T_n\}_{n=0}^N$ such that it is more accurate than the one obtained by $\G$.
The correction iteration (see also \cite{bal2006,gander2014,gander2007}) is defined as
\begin{equation}\label{scheme}
\begin{aligned}
u_{n}^{(0)}=&\G(T_{n},T_{n-1},u_{n-1}^{(0)})\\
u_{n}^{(k)}=&\G(T_{n},T_{n-1},u_{n-1}^{(k)})+\F(T_{n},T_{n-1},u_{n-1}^{(k-1)})-\G(T_{n},T_{n-1},u_{n-1}^{(k-1)}),\quad k\in\N_+
\end{aligned}
\end{equation}
starting from $u_0^{(k)}=u_0$ for all $k\in\N$.
The solution $\{u_n^{(k)}\}_{0\le n\le N}\subset H$ of \eqref{scheme} is obtained after the calculation of $\{u_n^{(k-1)}\}_{0\le n\le N}$, and is $\{\mathcal{B}_{T_n}\}_{0\le n\le N}$-adapted for any $k\in\N$.

\section{Parareal exponential $\theta$-scheme for the linear case}
\label{sec3}
This section is devoted to study parareal algorithms based on the exponential $\theta$-scheme for the following linear equation
\begin{align}\label{linearmodel}
du=(\bi\Delta u-\alpha u+\bi\lambda u)dt+Q^\frac12dW
\end{align}
with $\lambda\in\R$. We show that the proposed parareal algorithms are valid for longtime simulation with a unique invariant Gaussian distribution under some restrictions on $\theta\in[0,1]$.

Rewriting above equation through its components $u^{m}:=\langle u,e_m\rangle $,  we obtain
\begin{align*}
du^{m}=(-\lambda_m+\bi\lambda)u^{m}dt+\sum_{i=1}^\infty\langle Q^\frac12e_i,e_m\rangle d\beta_i,\quad m=1,\cdots,M.
\end{align*}
Its solution is given by an Ornstein--Uhlenbeck process
\begin{align*}
u^{m}(t)=e^{(-\lambda_m+\bi\lambda)t}u^{m}(0)+\sum_{i=1}^\infty\int_0^te^{(-\lambda_m+\bi\lambda)(t-s)}\langle Q^\frac12e_i,e_m\rangle d\beta_i(s)
\end{align*}
with $u^{m}(0)=\langle u_0,e_m\rangle$. 

\subsection{Complex invariant Gaussian measure}
Note that $\{u^{m}(t)\}_{t\ge0}$ satisfies a complex Gaussian distribution $\mathcal{N}(\bm,\bC,\bR)$ defined by its mean $\bm$, covariance $\bC$ and relation $\bR$:
\begin{align*}
\bm\left(u^{m}(t)\right):=&\E\left[u^{m}(t)\right]=e^{(-\lambda_m+\bi\lambda)t}\bm\left[u^{m}(0)\right],\\
\bC\left(u^{m}(t)\right):=&\E\left|u^{m}(t)-\bm\left(u^{m}(t)\right)\right|^2=e^{-2\alpha t}\bC\left(u^{m}(0)\right)+\frac{1-e^{-2\alpha t}}{\alpha}\|Q^\frac12e_m\|^2,\\
\bR\left(u^{m}(t)\right):=&\E\left(u^{m}(t)-\bm\left(u^{m}(t)\right)\right)^2=e^{2(-\lambda_m+\bi\lambda) t}\bR\left(u^{m}(0)\right).
\end{align*}
We use the notation $\mu^{m}_t:=\mathcal{N}(\bm(u^{m}(t)),\bC(u^{m}(t)),\bR(u^{m}(t)))$ for simplicity.

\begin{rk}\label{indep}
We consider a one-dimensional $\C$-valued Gaussian random variable $Z=\bf  a+\bi b$ with $\bf a$ and $\bf b$ being two $\R$-valued Gaussian random variables. If its relation vanishes, i.e.,
\begin{align*}
\bR(Z)=\E|{\bf a}-\E {\bf a}|^2-\E|{\bf b}-\E {\bf b}|^2+2\bi(\E[{\bf ab}]-\E {\bf a}\E {\bf b})=0,
\end{align*}
it implies $\E|{\bf a}-\E {\bf a}|^2=\E|{\bf b}-\E {\bf b}|^2$ and $\E[{\bf ab}]=\E {\bf a}\E {\bf b}$. Since $\bf a$ and $\bf b$ are both Gaussian, we obtain equivalently that $\bf a$ and $\bf b$ are independent with the same covariance.
\end{rk}
\begin{rk}\label{chara}
The characteristic function of a one-dimensional complex Gaussian variable $Z$ with distribution $\nu=\mathcal{N}(\bm,\bC,\bR)$ reads (see e.g. \cite{andersen1995})
\begin{align*}
\hat\nu(c):=&\E[\exp\{\bi\Re(\bar c Z)\}]=\int_{\C}\exp\{\bi\Re(\bar c z)\}\nu(dz)\\
=&\exp\left\{\bi\Re(\bar{c}\bm)-\frac14(\bar{c}\bC c+\Re(\bar{c}\bR\bar{c}))\right\},\quad c\in\C.
\end{align*}
It can be generalized for the infinite dimensional case utilizing inner product in $H$:
\begin{align*}
\hat\nu( w):=\exp\left\{\bi\Re\langle\bar w,\bm\rangle-\frac14\left(\langle\bC\bar w, w\rangle+\Re\langle\bR\bar w,\bar w\rangle\right)\right\},\quad  w\in H.
\end{align*}
\end{rk}

Hence, we get that the unique invariant measure of \eqref{linearmodel} is a complex Gaussian distribution, which is stated in the following theorem. We refer to \cite{debussche2005,EKZ17} and references therein for the existence of invariant measures for the nonlinear case, and refer to \cite{BV16,DZ96} and references therein for other types of SPDEs.
\begin{tm}\label{IM}
Assume that Assumption \ref{ap2} holds with $s=0$. 
The solution $u$ in \eqref{linearmodel} possesses a unique invariant measure 
\begin{align*}
\mu_{\infty}=\mathcal{N}\left(0,\frac1\alpha Q,0\right).
\end{align*} 
%More precisely, the distribution of $u^M$ converges to $\mu^M$ as $t\to\infty$. 
\end{tm}
\begin{proof}
Based on Remark \ref{indep}, we define 
\begin{align*}
u^{m}_{\infty}=\frac{\|Q^\frac12e_m\|}{\sqrt{2\alpha}}(\xi_m+\bi r_m)
\end{align*} 
with $\{\xi_m,r_m\}_{m\in\N}$ being independent standard $\R$-valued normal random variables, i.e., $\xi_m,r_m\sim\mathcal{N}(0,1)$. Apparently,
\begin{align*}
u^{m}_{\infty}\sim\mathcal{N}\left(0,\frac{\|Q^\frac12e_m\|^2}{\alpha},0\right)=:\mu^{m}_{\infty}.
\end{align*}

We claim that the following random variable has the distribution $\mu_{\infty}$: 
\begin{align*}
u_{\infty}:=\sum_{m=1}^\infty u^{m}_{\infty}e_m=\sum_{m=1}^\infty\frac{\|Q^\frac12e_m\|}{\sqrt{2\alpha}}(\xi_m+\bi r_m)e_m.
\end{align*}  
Compared with $u(t)=\sum_{m=1}^\infty u^{m}(t)e_m$,
it then suffices to show that the distribution $\mu^{m}_t$ of $u^m(t)$ converges to $\mu^{m}_{\infty}$.
As a result of Remark \ref{chara}, the characteristic function of $\mu^{m}_t$ is 
\begin{align*}
\hat\mu^{m}_t(c)=&\exp\Bigg\{\bi\Re(\bar ce^{(-\lambda_m+\bi\lambda)t}\E\left[u^{m}(0)\right])-\frac14\Re\left(e^{2(-\lambda_m+\bi\lambda) t}\bR\left(u^{m}(0)\right)\bar c^2\right)\\
&-\frac14\left(e^{-2\alpha t}\bC\left(u^{m}(0)\right)+\frac{1-e^{-2\alpha t}}{\alpha}\|Q^\frac12e_m\|^2|c|^2\right)\Bigg\}
\end{align*}
and $\hat\mu^{m}_t(c)\to\exp\{-\frac{\|Q^\frac12e_m\|^2}{4\alpha}|c|^2\}=\hat\mu^m_{\infty}(c).$
\end{proof}

\subsection{Parareal exponential $\theta$-scheme}
\label{sec3.2}
In this section, we construct a parareal algorithm based on the exponential $\theta$-scheme as the coarse propagator. We show that proposed parareal algorithm converges to the solution generated by the fine propagator $\F$ as $k\to\infty$.

We first define the exponential $\theta$-scheme applied to \eqref{linearmodel}:
\begin{align*}
u_n=S(\delta T)u_{n-1}+\bi(1-\theta)\lambda\delta TS(\delta T)u_{n-1}+\bi\theta\lambda\delta Tu_n+S(\delta T)Q^{\frac12}\delta_nW,
\end{align*}
or equivalently,
\begin{align}\label{linearG}
u_n=(1+\bi(1-\theta)\lambda\delta T)S_{\theta}S(\delta T)u_{n-1}+S_{\theta}S(\delta T)Q^{\frac12}\delta_nW=:\G_\theta(T_n,T_{n-1},u_{n-1})
\end{align}
with $S_\theta:=(1-\bi\theta\lambda\delta T)^{-1}$, $\theta\in[0,1]$ and $\delta_nW:=W(T_n)-W(T_{n-1})$.
The initial value of the numerical solution is the same as the initial value of the exact solution, and apparently $\{u_n\}_{n=0}^N$ is $\{\mathcal{B}_{T_n}\}_{n=0}^N$-adapted.

The distribution of $\{u_n\}_{n=0}^N$ can also be calculated in the same procedure as Theorem \ref{IM} by rewriting the Fourier components $u_{n}^m:=\langle u_{n},e_m\rangle$ of $u_{n}$ as
\begin{align*}
u_{n}^m=&(1+\bi(1-\theta)\lambda\delta T)S_\theta e^{-\lambda_m\delta T} u_{n-1}^m+S_\theta e^{-\lambda_m\delta T}\sum_{i=1}^\infty\langle Q^\frac12e_i,e_m\rangle\delta_{n}\beta_i\\
=&\eta^ne^{-\lambda_m\delta Tn}u_0^m
+S_\theta e^{-\lambda_m\delta T}\sum_{j=0}^{n-1}\eta^je^{-\lambda_m\delta Tj}\sum_{i=1}^\infty\langle Q^\frac12e_i,e_m\rangle\delta_{n-j}\beta_i
\end{align*}
with 
\begin{align*}
\eta:=(1+\bi(1-\theta)\lambda\delta T)S_\theta=\frac{1+\bi(1-\theta)\lambda\delta T}{1-\bi\theta\lambda\delta T}.
\end{align*}
Then according to the independence of $\{\delta_{n-j}\beta_i\}_{1\le j\le n-1,i\ge1}$ and $\E|\delta_{n-j}\beta_i|^2=2\delta T$, we derive the distribution of $u_{n}^m$ defined by its mean, covariance and relation:
\begin{align*}
\bm(u_{n}^m)=&\eta^{n}e^{-\lambda_m\delta Tn}\E[u_0^m],\\
\bC(u_{n}^m)=&|\eta|^{2n}e^{-2\alpha\delta Tn}\bC(u_0^m)\\
&+\left(\left(1+\theta^2\lambda^2\delta T^2\right)e^{2\alpha\delta T}\right)^{-1}\frac{1-\tilde\eta^n}{1-\tilde\eta}\|Q^\frac12e_m\|^2(2\delta T),\\
\bR(u_{n}^m)=&\eta^{2n}e^{-2\lambda_m\delta Tn}\bR(u_0^m),
\end{align*}
where 
\begin{align*}
\tilde\eta:=\frac{1+(1-\theta)^2\lambda^2\delta T^2}{\left(1+\theta^2\lambda^2\delta T^2\right)e^{2\alpha\delta T}}=|\eta|^2e^{-2\alpha\delta T}
\end{align*}
is called the {\it stable function} here.
\begin{figure}[h]
\centering
\subfigure{
\begin{minipage}[t]{0.31\linewidth}
  \includegraphics[height=3.5cm,width=4cm]{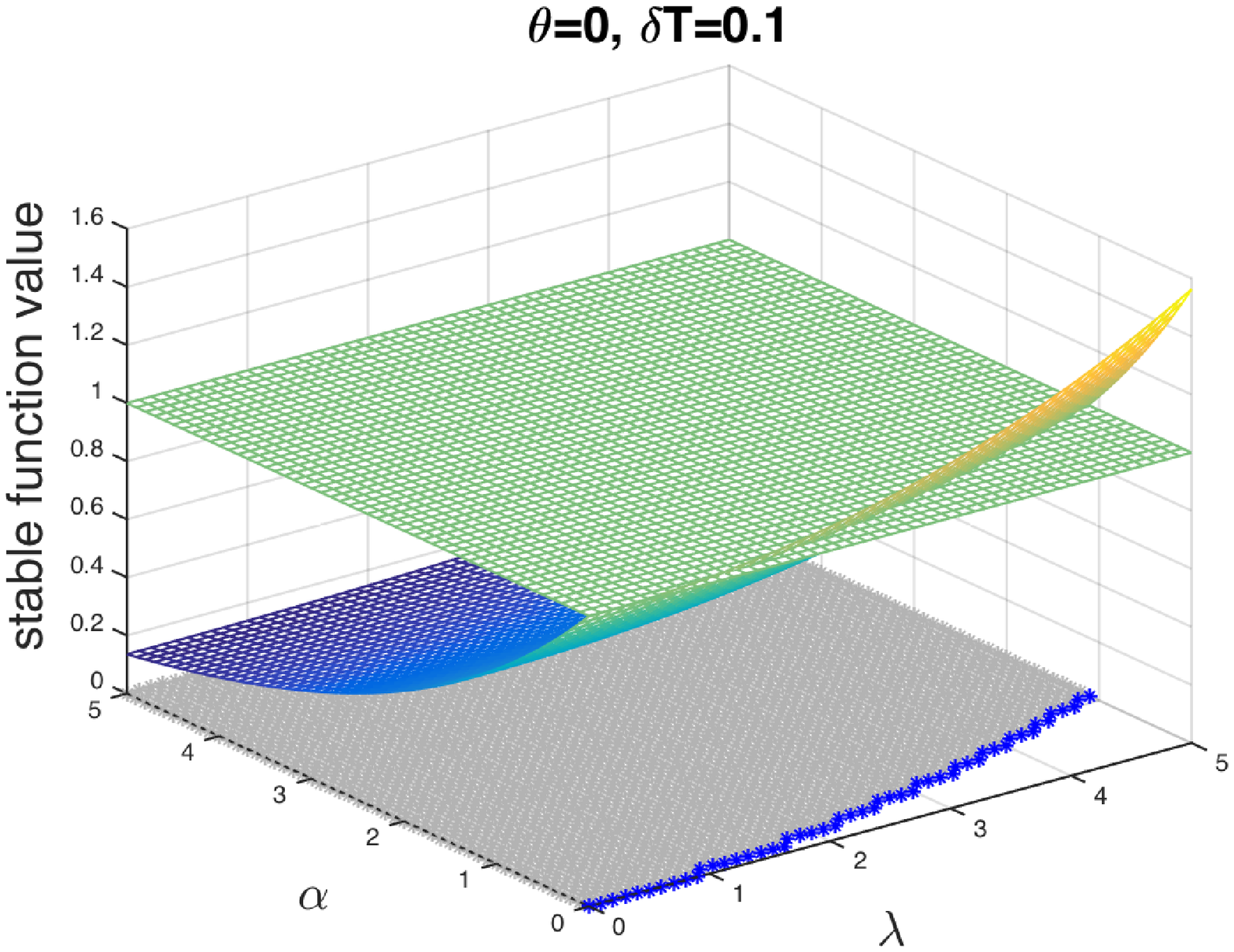}
  \end{minipage}
  }
  \subfigure{
  \begin{minipage}[t]{0.306\linewidth}
  \includegraphics[height=3.5cm,width=4cm]{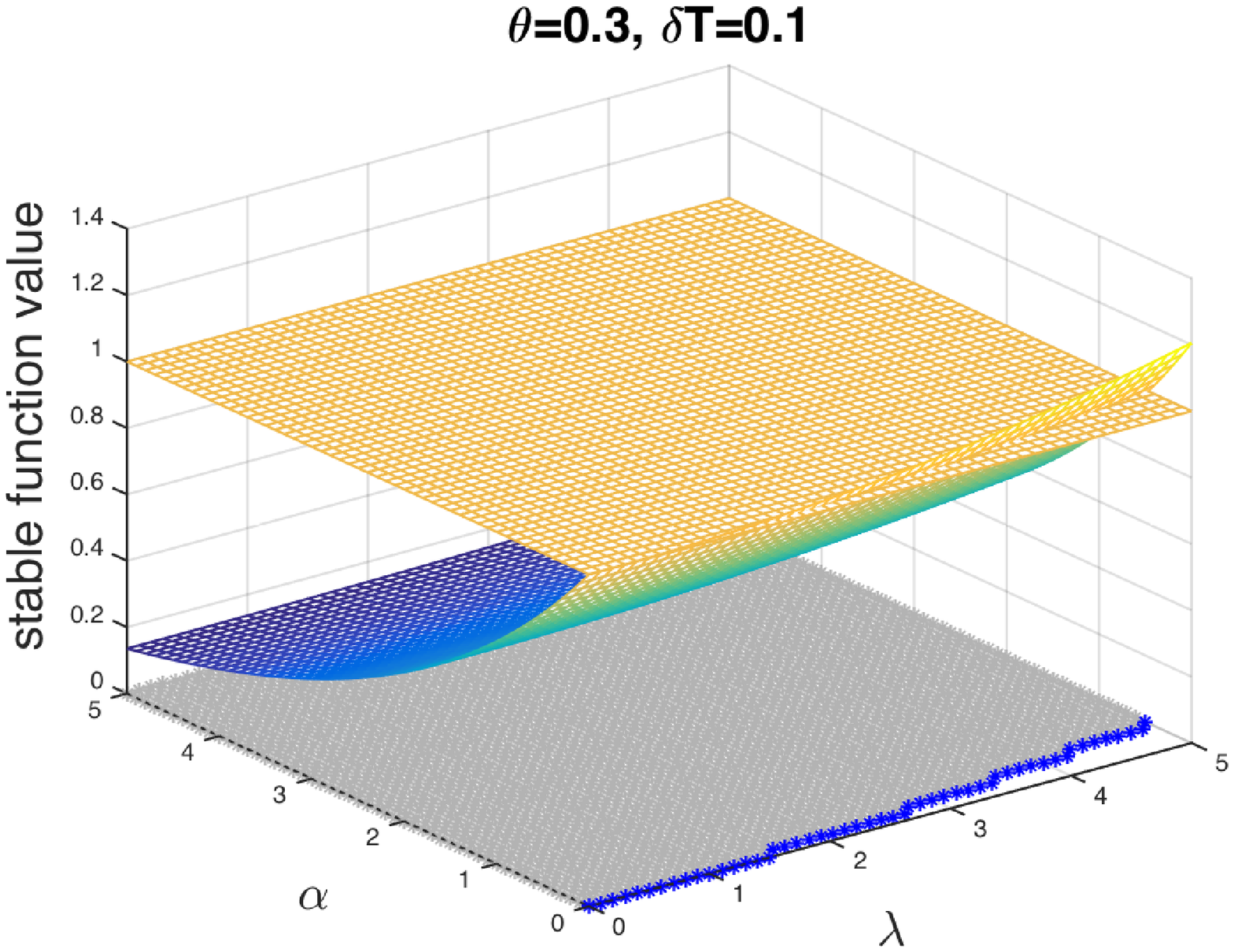}
  \end{minipage}
  }
    \subfigure{
  \begin{minipage}[t]{0.31\linewidth}
  \includegraphics[height=3.5cm,width=4cm]{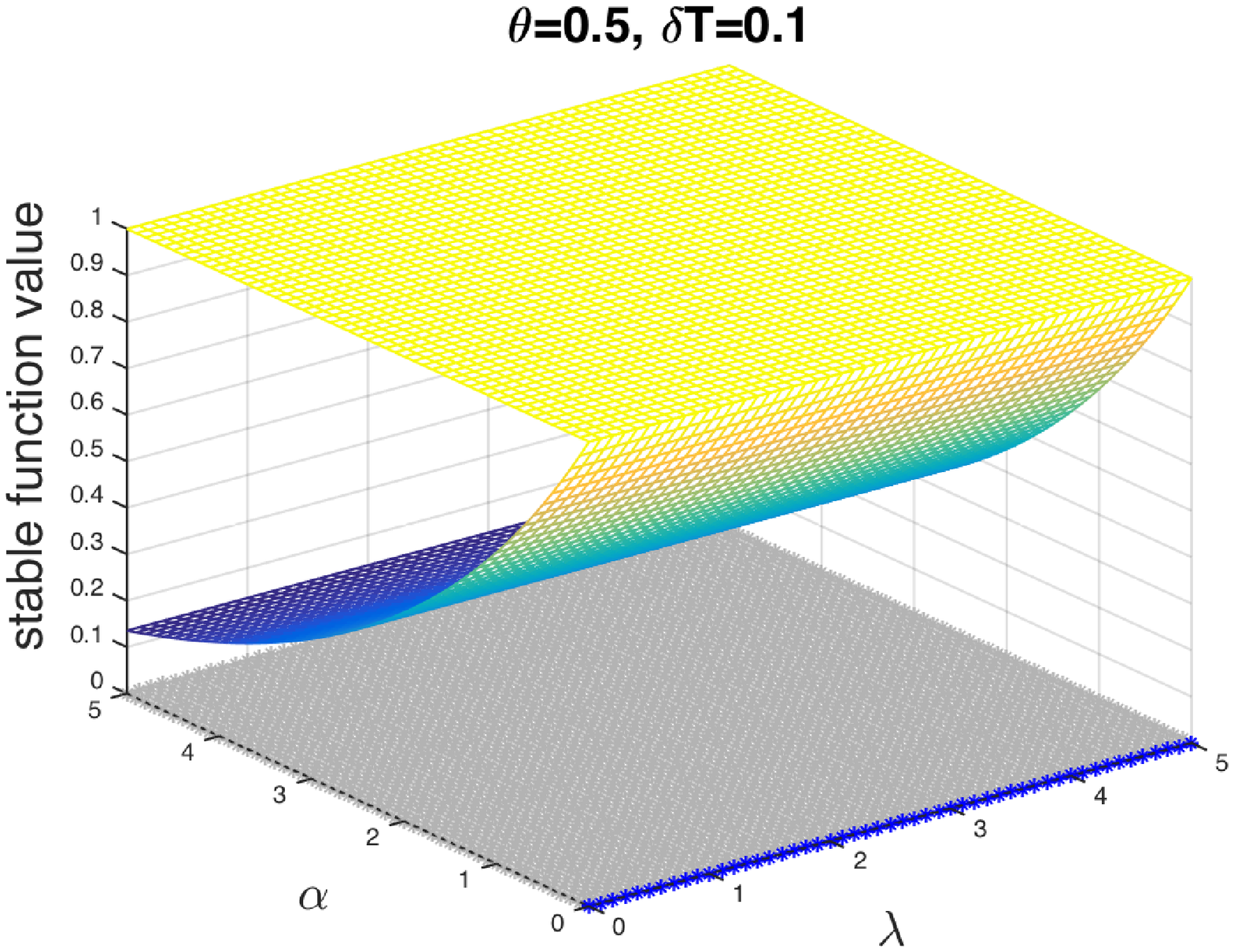}
  \end{minipage}
  }\\
  \subfigure{
\begin{minipage}[t]{0.31\linewidth}
  \includegraphics[height=3.5cm,width=4cm]{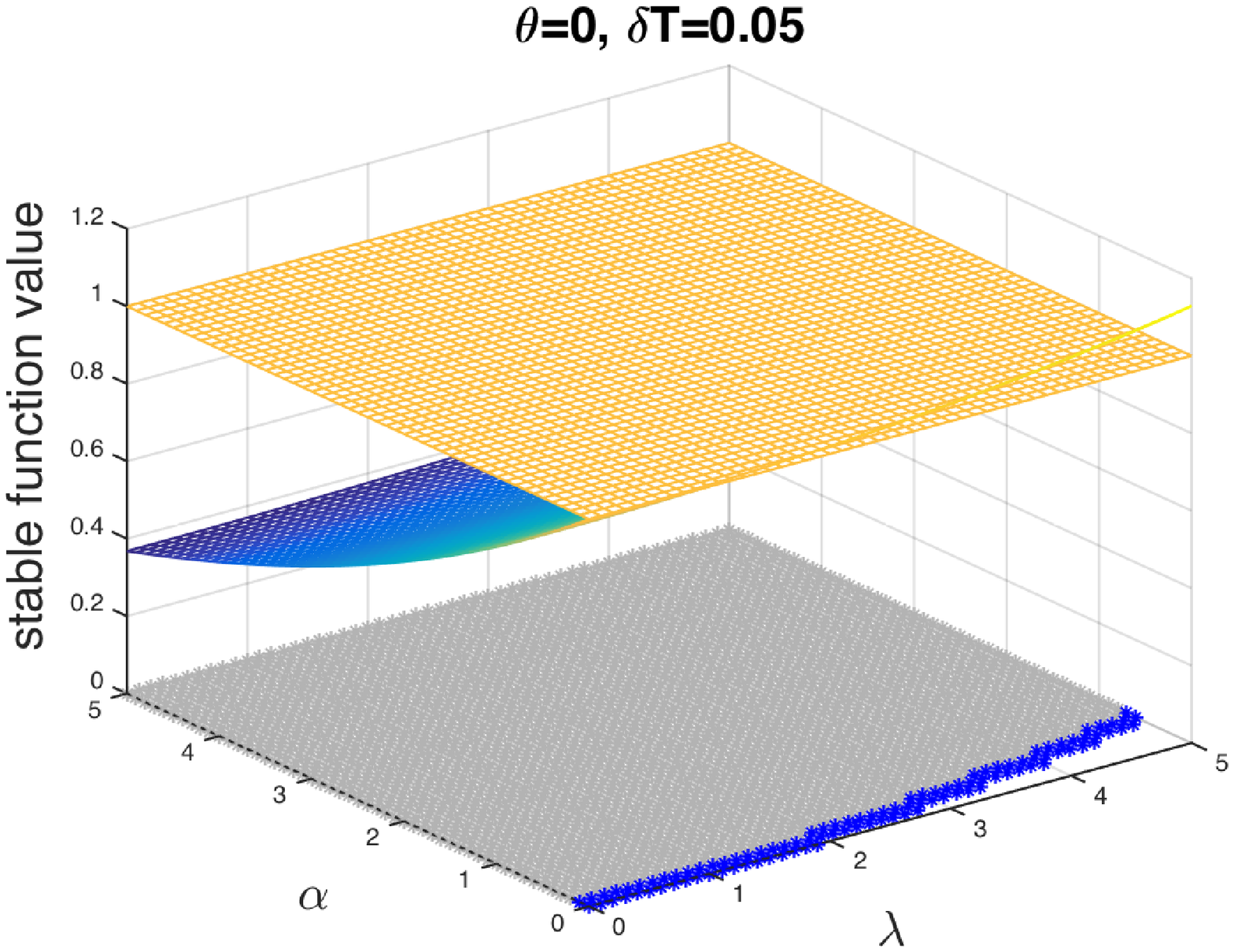}
  \end{minipage}
  }
  \subfigure{
  \begin{minipage}[t]{0.306\linewidth}
  \includegraphics[height=3.5cm,width=4cm]{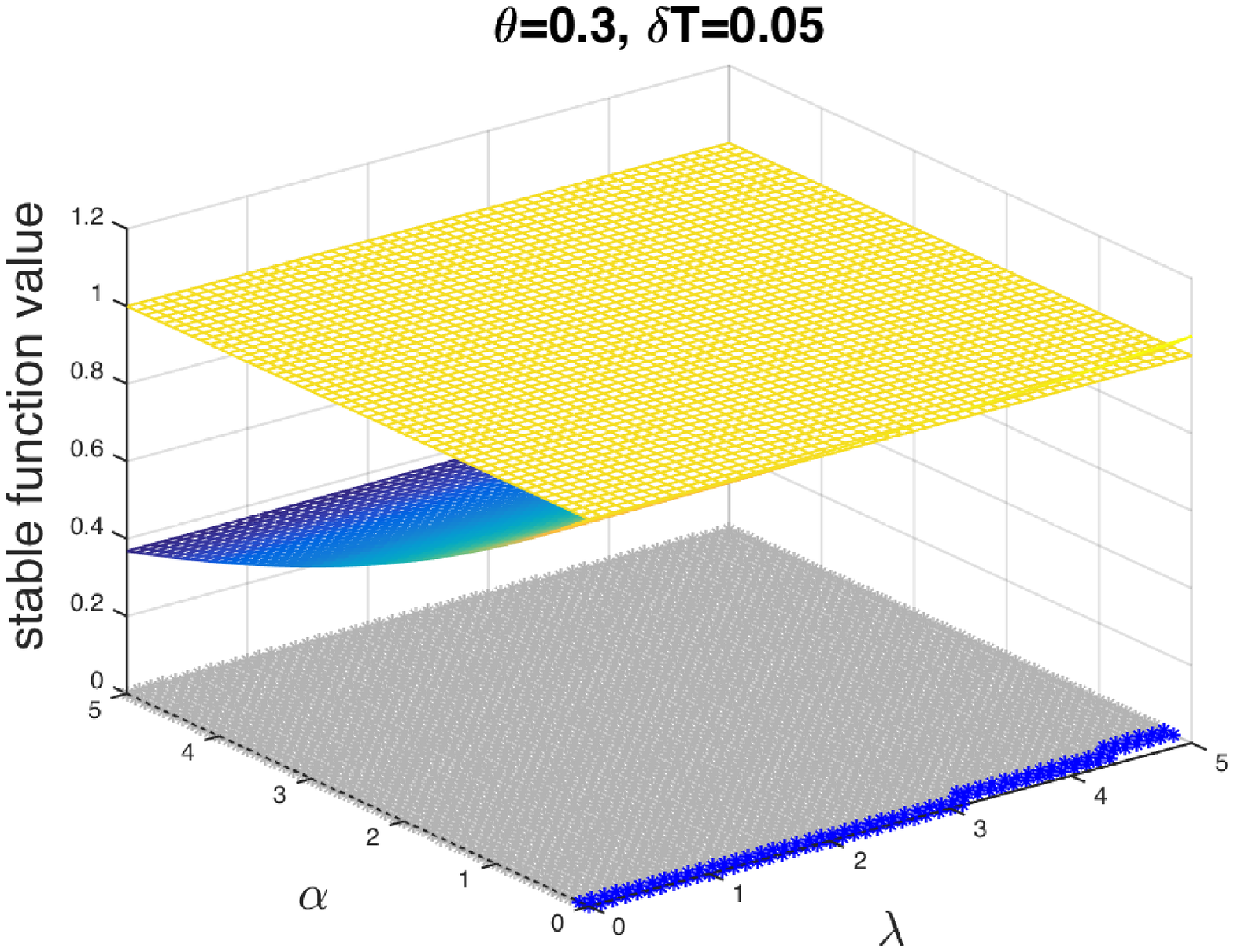}
  \end{minipage}
  }
    \subfigure{
  \begin{minipage}[t]{0.31\linewidth}
  \includegraphics[height=3.5cm,width=4cm]{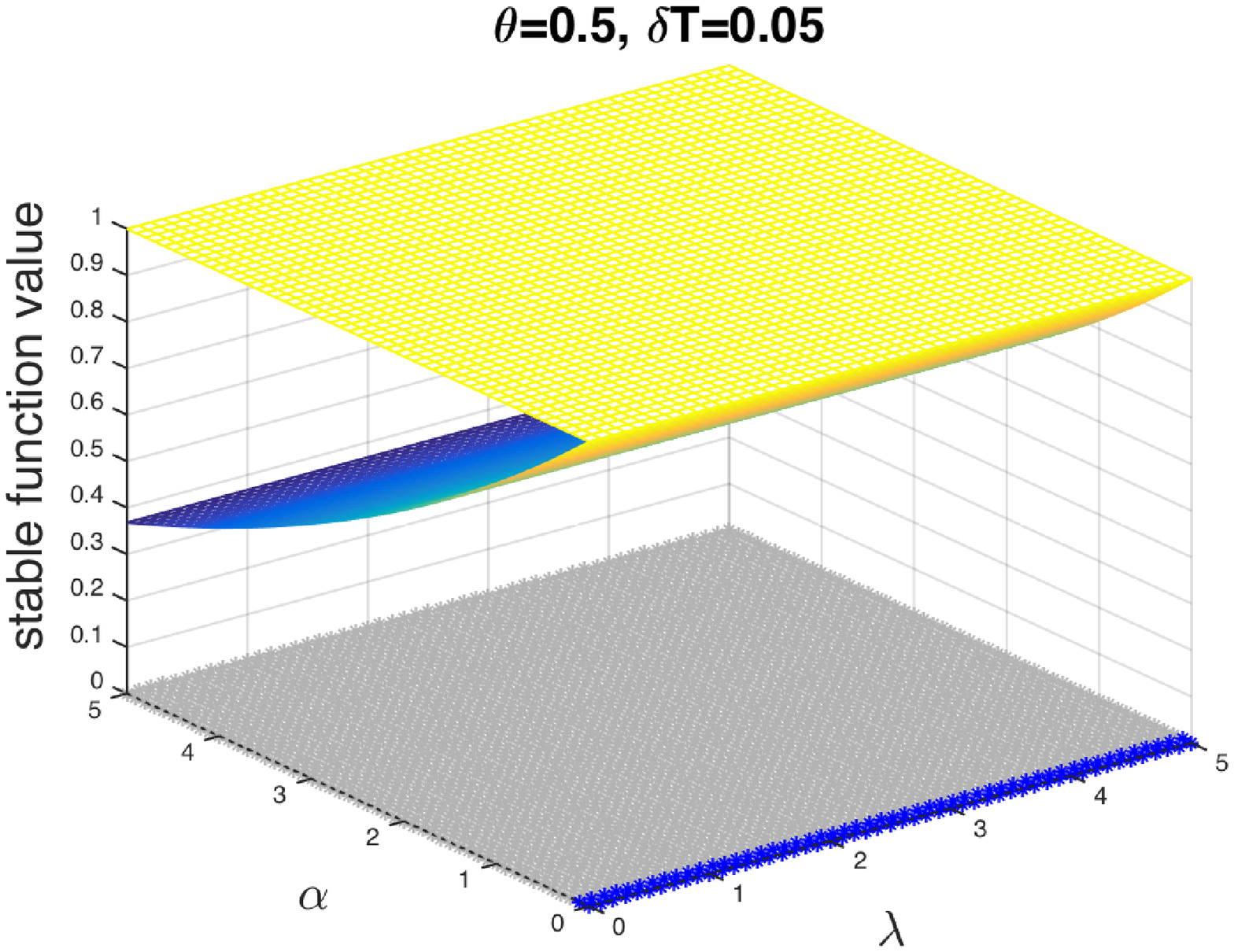}
  \end{minipage}
  }
  \caption{Convergence area (grey) vs. $\alpha$ and $\lambda$.}
\label{figstable}
\end{figure}

The distribution of $u_n^m$ converges to $\mu^m_\infty$ as $n\to\infty$ and $\delta T\to0$ for any $\alpha>0$ if and only if $|\eta|<1$, or equivalently, $\theta\in[\frac12,1]$, see Figure \ref{figstable}. The surface in each subfigures in Figure \ref{figstable} denotes the stable function for different $\theta=0,0.3,0.5$ and $\delta T=0.1,0.005$.
This condition also leads to the time-independent error analysis of the parareal algorithm, see Theorem \ref{tm3.2}.

%\begin{algorithm}
%\caption{Parareal exponential $\theta$-scheme}
%\end{algorithm}

The parareal algorithm \eqref{scheme} with $\G_\theta$ being the coarse propagator is expressed as
\begin{align}\label{linearP}
u_n^{(k)}=&(1+\bi(1-\theta)\lambda\delta T)S_\theta S(\delta T)u_{n-1}^{(k)}\nonumber\\
&-(1+\bi(1-\theta)\lambda\delta T)S_\theta S(\delta T)u_{n-1}^{(k-1)}
+\F(T_n,T_{n-1},u_{n-1}^{(k-1)})\nonumber\\
=&\eta S(\delta T)u_{n-1}^{(k)}-\eta S(\delta T)u_{n-1}^{(k-1)}
+\F(T_n,T_{n-1},u_{n-1}^{(k-1)}).
\end{align}

The following result gives the error caused by the parareal algorithms. When the coarse step size $\delta T$ is not extremely small, the convergence shows order $k$ with respect to $\delta T$ in a strong sense.
\begin{tm}\label{tm3.2}
Let Assumptions \ref{ap1} and \ref{ap2} hold with $s=0$, and $\{u_n^{(k)}\}_{0\le n\le N,k\in\N}$ be the solution of \eqref{linearP} with $\F$ being the exact propagator. Assume that $\lambda\delta T<1.$ Then for a fixed iteration step $k\in\N$, $u_n^{(k)}$ is an approximation of $u(T_n)$ with order $k$. More precisely, if $\alpha>\sqrt{\left(\frac12-\theta\right)^+}|\lambda|$, then
\begin{align*}
\sup_{n\in\N}\left\|u(T_n)-u_n^{(k)}\right\|_{L^2(\Omega;H)}\le C\left[(2\theta-1)\delta T^k+\delta T^{2k}\right]\sup_{n\in\N}\left\|u(T_n)-u_n^{(0)}\right\|_{L^2(\Omega;H)}
\end{align*}
with $C=C(k,\alpha,\theta,\lambda)$ independent of time interval. 
Here, $\left(\frac12-\theta\right)^+:=\left(\frac12-\theta\right)\vee0$. 

Otherwise,
\begin{align*}
\sup_{0\le n\le N}\left\|u(T_n)-u_n^{(k)}\right\|_{L^2(\Omega;H)}\le C\delta T^k\sup_{0\le n\le N}\left\|u(T_n)-u_n^{(0)}\right\|_{L^2(\Omega;H)}
\end{align*}
with $C=C(T_N,k)$ and $T_N=\delta TN$ for some fixed $N\in\N$.
\end{tm}
\begin{proof}
The parareal algorithm based on $\G_\theta$ with $\F$ denoting the exact propagator yields
\begin{align*}
u_{n}^{(k)}=&\eta S(\delta T)u_{n-1}^{(k)}-\eta S(\delta T)u_{n-1}^{(k-1)}+S(\delta T)u_{n-1}^{(k-1)}\\
&+\bi\lambda\int_{T_{n-1}}^{T_n}S(T_n-s)u_{u_{n-1}^{(k-1)}}(s)ds+\int_{T_{n-1}}^{T_n}S(T_n-s)Q^\frac12dW
\end{align*}
with $n\ge1,k\ge1$ and $u_{u_{n-1}^{(k-1)}}(s)$ denoting the exact solution at time $s$ starting from $u_{n-1}^{(k-1)}$ at time $T_{n-1}$. 

Denoting $\epsilon_n^{(k)}:=u_n^{(k)}-u(T_n)$, we obtain
\begin{align*}
\epsilon_n^{(k)}=&\eta S(\delta T)\epsilon_{n-1}^{(k)}-\eta S(\delta T)\epsilon_{n-1}^{(k-1)}+S(\delta T)\epsilon_{n-1}^{(k-1)}\\
&+\bi\lambda\int_{T_{n-1}}^{T_n}S(T_n-s)\left[u_{u_{n-1}^{(k-1)}}(s)-u_{u(T_{n-1})}(s)\right]ds\\
=&\eta S(\delta T)\epsilon_{n-1}^{(k)}+\left[e^{\bi\lambda\delta T}-\eta\right]S(\delta T)\epsilon_{n-1}^{(k-1)},
\end{align*}
where in the last step we have used the following fact
\begin{align*}
&u_{u_{n-1}^{(k-1)}}(s)-u_{u(T_{n-1})}(s)\\
=&e^{(\bi\Delta-\alpha+\bi\lambda)(s-T_{n-1})}u_{n-1}^{(k-1)}+\int_{T_{n-1}}^se^{(\bi\Delta-\alpha+\bi\lambda)(s-r)}Q^\frac12dW(r)\\
&-e^{(\bi\Delta-\alpha+\bi\lambda)(s-T_{n-1})}u(T_{n-1})-\int_{T_{n-1}}^se^{(\bi\Delta-\alpha+\bi\lambda)(s-r)}Q^\frac12dW(r)\\
=&S(s-T_{n-1})e^{\bi\lambda(s-T_{n-1})}\epsilon_{n-1}^{(k-1)}.
\end{align*}
Hence, we get
\begin{align}\label{epsilonk}
\|\epsilon_n^{(k)}\|_{L^2(\Omega;H)}\le&|\eta|e^{-\alpha\delta T}\|\epsilon_{n-1}^{(k)}\|_{L^2(\Omega;H)}+|e^{\bi\lambda\delta T}-\eta|e^{-\alpha\delta T}\|\epsilon_{n-1}^{(k-1)}\|_{L^2(\Omega;H)}\nonumber\\
\le&\left(|\eta|e^{-\alpha\delta T}\right)^n\|\epsilon_{0}^{(k)}\|_{L^2(\Omega;H)}\nonumber\\
&+|e^{\bi\lambda\delta T}-\eta|e^{-\alpha\delta T}\sum_{j=0}^{n-1}\left(|\eta|e^{-\alpha\delta T}\right)^{n-1-j}\|\epsilon_{j}^{(k-1)}\|_{L^2(\Omega;H)}\nonumber\\
=&|e^{\bi\lambda\delta T}-\eta|e^{-\alpha\delta T}\sum_{j=1}^{n-1}\left(|\eta|e^{-\alpha\delta T}\right)^{n-1-j}\|\epsilon_{j}^{(k-1)}\|_{L^2(\Omega;H)}
\end{align}
based on the fact $\epsilon_0^{(k)}=0$ for any $k\in\N.$
Denoting the error vector 
\begin{align*}
\varepsilon^{(k)}:=\left(\|\epsilon_1^{(k)}\|_{L^2(\Omega;H)},\cdots,\|\epsilon_n^{(k)}\|_{L^2(\Omega;H)}\right)^\top
\end{align*} 
and the $n$-dimensional matrix (see also \cite{gander2007})
\begin{equation*}M(\beta)=
\left(
\begin{array}{cccccc}
0&0&\cdots &0& 0\\
1&0&\cdots & 0&0\\
\beta&1&\cdots &0&0\\
\beta^{2}&\beta&\cdots&0&0\\
\vdots & \vdots& \vdots &\vdots&\vdots \\
\beta^{n-2}&\beta^{n-3}&\cdots&1&0
\end{array}
\right),
\end{equation*}
we can rewrite \eqref{epsilonk} as
\begin{align*}
\varepsilon^{(k)}\le |e^{\bi\lambda\delta T}-\eta|e^{-\alpha\delta T}M(|\eta|e^{-\alpha\delta T})\varepsilon^{(k-1)}
\le |e^{\bi\lambda\delta T}-\eta|^ke^{-\alpha\delta Tk}M^k(|\eta|e^{-\alpha\delta T})\varepsilon^{(0)}.
\end{align*}
It is shown in \cite{gander2007} that  
\begin{equation*}
\|M^k(\beta)\|_\infty\le
\left\{
\begin{aligned}
&\min\left\{\left(\frac{1-\beta^{n-1}}{1-\beta}\right)^k,\left(
\begin{array}{c}
n-1\\
k
\end{array}
\right)\right\}\quad \text{if}~\beta<1,\\
&\beta^{n-1-k}\left(
\begin{array}{c}
n-1\\
k
\end{array}
\right)\quad\quad\quad\quad\quad\quad\quad\quad \text{if}~\beta\ge1,
\end{aligned}
\right.
\end{equation*}
where
\begin{align*}
\left(
\begin{array}{c}
n-1\\
k
\end{array}
\right)
=\frac{(n-1)(n-2)\cdots(n-k)}{k!}
\le\frac{n^k}{k!}.
\end{align*}

If $\alpha>\sqrt{\left(\frac12-\theta\right)^+}|\lambda|$, we get
\begin{align*}
e^{2\alpha\delta T}>1+2\alpha^2\delta T^2>1+(1-2\theta)^+\lambda^2\delta T^2
>1+\frac{(1-2\theta)\lambda^2\delta T^2}{1+\theta^2\lambda^2\delta T^2}
=|\eta|^2,
\end{align*}
which then yields $|\eta|e^{-\alpha\delta T}<1$. 
 It is apparent that this condition holds for all $\alpha>0$ if $\theta\in[\frac12,1]$. We conclude under this condition that
\begin{align*}
\|\varepsilon^{(k)}\|_\infty\le\left(\frac{|e^{\bi\lambda\delta T}-\eta|e^{-\alpha\delta T}}{1-|\eta|e^{-\alpha\delta T}}\right)^k\|\varepsilon^{(0)}\|_\infty.
\end{align*}
The solution of \eqref{linearP} with $\F$ being the exact flow converges to the exact solution as $k\to\infty$ if
\begin{align*}
|e^{\bi\lambda\delta T}-\eta|e^{-\alpha\delta T}+|\eta|e^{-\alpha\delta T}<1.
\end{align*}
For some fixed $k\in\N$, we get through Taylor expansion that
\begin{align*}
|e^{\bi\lambda\delta T}-\eta|^ke^{-\alpha\delta Tk}
\le\left(\frac12(2\theta-1)\lambda^2\delta T^2+C\delta T^3\right)^ke^{-\alpha\delta Tk},
\end{align*}
and in addition
\begin{align*}
\|M^k(|\eta|e^{-\alpha\delta T})\|_\infty\le(1-|\eta|e^{-\alpha\delta T})^{-k}
\le\left(1-e^{\left(\sqrt{\left(\frac12-\theta\right)^+}|\lambda|-\alpha\right)\delta T}\right)^{-k}
\le(C\delta T^{-1})^k,
\end{align*}
where above constant $C=C\left(\alpha-\sqrt{\left(\frac12-\theta\right)^+}|\lambda|\right)$ decreases as $\alpha-\sqrt{\left(\frac12-\theta\right)^+}|\lambda|$ becomes larger. 
Eventually, we conclude
\begin{align*}
\|\varepsilon^{(k)}\|_\infty\le (C(2\theta-1)\delta T+C\delta T^2)^k\|\varepsilon^{(0)}\|_\infty.
\end{align*}

If $\theta\in[0,\frac12)$ and $\alpha\le\sqrt{\left(\frac12-\theta\right)}|\lambda|$, we revise above proof as
\begin{align*}
\|\varepsilon^{(k)}\|_\infty
\le&\left(|e^{\bi\lambda\delta T}-\eta|e^{-\alpha\delta T}\right)^k(|\eta|e^{-\alpha\delta T}\vee1)^{n-1-k}\frac{n^k}{k!}\|\varepsilon^{(0)}\|_\infty\\
\le&\left(C\delta T^2e^{-\alpha\delta T}\right)^ke^{\left(\sqrt{2(1-2\theta)}|\lambda|-\alpha\right) T_n}\frac{n^k}{k!}\|\varepsilon^{(0)}\|_\infty\\
\le&\frac{(CT_ne^{-\alpha\delta T})^k}{k!}e^{\left(\sqrt{2(1-2\theta)}|\lambda|-\alpha\right) T_n}\delta T^k\|\varepsilon^{(0)}\|_\infty,
\end{align*}
which converges as $k\to\infty$ and shows order $k$ only on finite time intervals.
\end{proof}
\begin{rk}\label{rk3}
Note that the Fourier components of the noise term 
\begin{align*}
\int_0^te^{-\lambda_m(t-s)}\sum_{i=1}^\infty\langle Q^\frac12e_i,e_m\rangle d\beta_i(s),\quad m\in\N
\end{align*}
are Gaussian processes and their increments can be simulated through random variables in the same distribution.
Hence, scheme \eqref{linearG} can also be replaced by
\begin{align*}
u_n=(1+\bi(1-\theta)\lambda\delta T)S_{\theta}S(\delta T)u_{n-1}+S_{\theta}\int_{T_{n-1}}^{T_n}S(T_n-s)Q^{\frac12}dW(s),
\end{align*}
and the accuracy of parareal algorithm \eqref{linearP} remains the same. 
\end{rk}
\begin{rk}\label{highregularity}
If instead, the implicit Euler scheme is considered as the coarse propagator $\G$, the parareal algorithm \eqref{scheme} with $\F$ being the exact propagator turns to be
\begin{align*}
u_n^{(k)}=\breve S_{\delta T} u_{n-1}^{(k)}-\breve S_{\delta T} u_{n-1}^{(k-1)}+\breve S(\delta T) u_{n-1}^{(k-1)}+\int_{T_{n-1}}^{T_n}\breve S(T_n-s)Q^{\frac12}dW(s),
\end{align*}
where
$\breve S_{\delta T}=\left(1+\alpha\delta T-\bi\lambda\delta T-\bi\delta T\Delta\right)^{-1}$ and $\breve S(\delta T)=e^{(\bi\Delta-\alpha+\bi\lambda)\delta T}$.

In this case, the error between $u_n^{(k)}$ and $u(T_n)$ shows
\begin{align*}
\epsilon_n^{(k)}=\breve S_{\delta T}\epsilon_{n-1}^{(k)}+\left(\breve S(\delta T)-\breve S_{\delta T}\right)\epsilon_{n-1}^{(k-1)}.
\end{align*}
To gain a convergence order, the estimations of $\|\breve S(\delta T)-\breve S_{\delta T}\|_{\mathcal{L}(\dot H^{s},H)}$ and $\|\epsilon_n^{(0)}\|_{\dot H^{ks}}$ will be needed. It then requires a extremely high regularity of both $u(t)$ and $u_n^{(0)}$, and that parameter $s$ in Assumption \ref{ap2} is large enough, while it is not proper to give such regularity assumptions.
\end{rk}

\section{Application to the nonlinear case}
\label{sec4}
For the nonlinear case \eqref{model}, 
parareal exponential $\theta$-scheme is also suitable for longtime simulation with some restriction on $\delta T$ and $\alpha$. We take the case $\theta=0$ as a keystone to show the convergence of the proposed parareal algorithm  and its fully discrete scheme with $\F$ being a numerical propagator.

Moreover, to ensure that less restriction on $\delta T$ is needed, some modification of the coarse propagator is required instead of using the exponential $\theta$-scheme. We give the convergence condition for the modified exponential $\theta$-scheme with general $\theta\in[0,1]$.

\subsection{Parareal exponential Euler scheme($\theta=0$)}
We define the coarse propagator based on the exponential Euler scheme
\begin{align}\label{nonlinearG}
u_{n+1}=S(\delta T)u_n+\bi S(\delta T)F(u_n)\delta T+S(\delta T)Q^{\frac12}\delta_{n+1}W
=:\G_{I}(T_{n+1},T_n,u_n)
\end{align}
with $\delta_{n+1}W:=W(T_{n+1})-W(T_n)$. The initial value of the numerical solution is the same as the initial value of the exact solution, and apparently $\{u_n\}_{n=1}^N$ is $\{\mathcal{B}_{T_n}\}_{n=1}^N$-adapted.

The following result gives the error caused by the parareal algorithms. When the coarse step size $\delta T$ is not extremely small, the convergence shows order $k$ with respect to $\delta T$ in a strong sense. Its proof is quite similar to that of Theorem \ref{tm3.2} and is given in the Appendix.
\begin{tm}\label{tm4.1}
Let Assumptions \ref{ap1} and \ref{ap2} hold with $s=0$, and $\{u_n^{(k)}\}_{0\le n\le N,k\in\N}$ be the solution of \eqref{scheme} with $\F$ being the exact propagator and $\G=\G_{I}$ being the propagator defined in \eqref{nonlinearG}. Then for $\alpha\ge0$ and any $1\le n\le N$, $u_n^{(k)}$ converges to $u(T_n)$ as $k\to\infty$. More precisely, 
\begin{align*}
&\sup_{0\le n\le N}\left\|u(T_n)-u_n^{(k)}\right\|_{L^2(\Omega;H)}\\
\le &\left(e^{-\alpha\delta T}\right)^k\frac{(CT)^k}{k!}\left(e^{( L_F-\alpha)T}\vee1\right)\sup_{0\le n\le N}\left\|u(T_n)-u_n^{(0)}\right\|_{L^2(\Omega;H)}
\end{align*}
for any $k\in\N$ with some positive constant $C$ depending only on $L_F$ and $\alpha$.

If $\alpha>0$, there exists some $\delta T_*=\delta T_*(\alpha)\in(0,1)$ satisfying $\delta T_*^{-1}\ln\delta T_*^{-1}=\alpha$ such that the error above shows order $k$ with respect to $\delta T$ when $\delta T\in[\delta T_*,1)$:
\begin{align*}
&\sup_{0\le n\le N}\left\|u(T_n)-u_n^{(k)}\right\|_{L^2(\Omega;H)}\\
\le& \left(\delta T\right)^k\frac{(CT)^k}{k!}\left(e^{(L_F-\alpha)T}\vee1\right)\sup_{0\le n\le N}\left\|u(T_n)-u_n^{(0)}\right\|_{L^2(\Omega;H)}.
\end{align*}
\end{tm}

To obtain an implementable numerical method, the fine propagator $\F$ need to be chosen as a proper numerical method instead of the exact propagator.
In this case, it is called a fully discrete scheme, which does not mean the discretization in both space and time direction as it usually does. 
We refer to \cite{wang2017} for the discretization in space of stochastic cubic nonlinear Schr\"odinger equation, which is also available for the model considered in the present paper.

In particular, we choose $\F$ as a propagator obtained by applying the exponential integrator repeatedly on the fine grid with step size $\delta t$:
\begin{align*}
\F_{I}(t_{n,j},t_{n,j-1},v):=S(\delta t)v+\bi S(\delta t)F(v)\delta t+S(\delta t)Q^{\frac12}\delta_{n,j}W,\quad \forall~v\in H
\end{align*}
with $\delta_{n,j}W:=W(t_{n,j})-W(t_{n,j-1})$. 
Hence, we get the following fully discrete scheme:
\begin{equation}\label{full}
\begin{aligned}
&u_{n+1}^{(0)}=\G_{I}(T_{n+1},T_n,u_n^{(0)}),\quad u_0^{(0)}=u_0,\quad n=0,\cdots,N-1,\\
&\hat u_{n,j}^{(k-1)}=\F_{I}(t_{n,j},t_{n,j-1},\hat u_{n,j-1}^{(k-1)}),\quad \hat u_{n,0}^{(k-1)}=u_n^{(k-1)},\quad j=1,\cdots, J,\quad k\in\N\backslash\{0\},\\
&u_{n+1}^{(k)}=\G_{I}(T_{n+1},T_n,u_n^{(k)})+\hat u_{n,J}^{(k-1)}-\G_{I}(T_{n+1},T_n,u_n^{(k-1)}),\quad k\in\N\backslash\{0\},
\end{aligned}
\end{equation}
where the notation $t_{n,j}$ has been defined in Section \ref{sec2}.

%\subsection{Convergence analysis}
The approximate error of the fully discrete scheme \eqref{full} comes from two parts: the parareal technique based on a coarse propagator and the approximate error of the fine propagator.
In fact, the second part is exactly the approximate error of a specific serial scheme without iteration and depends heavily on the regularity of the noise given in Assumption \ref{ap2}, which will not be dealt with here. The readers are referred to \cite{wang2017,DD04,DD06} and references therein for the study on accuracy of serial schemes. 
We now focus on the error caused by the former part
and aim to show that the solution of \eqref{full} converges to the solution of the fine propagator $\F$ as $k$ goes to infinity. 
To this end, we denote by
\begin{align*}
v_{n,j}=\F_{I}(t_{n,j},t_{n,j-1},v_{n,j-1}),\quad n=0,\cdots,N,~ j=1,\cdots,J
\end{align*}
the solution of $\F$ on fine gird $\{t_{n,j}\}_{n\in\{0,\cdots,N\},j\in\{0,\cdots,J\}}$ starting from $v_{0,0}=u_0$, where $t_{n+1,0}=T_{n+1}=t_{n,J}$ and $v_{n+1,0}:=v_{n,J}$.
\begin{tm}\label{tm4.2}
Let Assumptions \ref{ap1} and \ref{ap2} hold with $s=0$ and $\{u_n^{(k)}\}_{0\le n\le N,k\in\N}$ be the solution of \eqref{full}. Then for any $k\in\N$, it holds
\begin{align*}
&\sup_{0\le n\le N}\left\|u^{(k)}_n-v_{n,0}\right\|_{L^2(\Omega;H)}\\
\le& \left(e^{-\alpha\delta T}\right)^k\frac{(CT)^k}{k!}\left(e^{(L_F-\alpha)T}\vee1\right)\sup_{0\le n\le N}\left\|u_n^{(0)}-v_{n,0}\right\|_{L^2(\Omega;H)}.
\end{align*}

In addition, if $\delta T\in[\delta T_*,1)$ with $\delta T_*$ being defined as in Theorem \ref{tm4.1}, the error shows order $k$ with respect to $\delta T$ similar to that in Theorem \ref{tm4.1}.
\end{tm}

The proof of this theorem follows the same procedure as that of Theorem \ref{tm4.1} and is given in the Appendix for the readers' convenience.

\subsection{Parareal exponential $\theta$-scheme over longtime}
We now consider the exponential $\theta$-scheme in the nonlinear case
\begin{align*}
u_n=&S(\delta T)u_{n-1}+\bi(1-\theta)\delta TS(\delta T)F(u_{n-1})
+\bi\theta\delta TF(u_n)+S(\delta T)Q^\frac12\delta_nW.
\end{align*} 
The existence and uniqueness of the numerical solution is obtained under Assumptions \ref{ap1} and \ref{ap2} through the same procedure as those in \cite{wang2017,DD04}.
So we denote the unique solution of above scheme by $u_n=\tilde\G_\theta(T_n,T_{n-1},u_{n-1})$.

The parareal algorithm based on $\tilde\G_{\theta}$ with $\F$ denoting the exact propagator can be expressed as
\begin{align}\label{solu3}
u_n^{(k)}=&\tilde\G_{\theta}(T_n,T_{n-1},u_{n-1}^{(k)})+\F(T_n,T_{n-1},u_{n-1}^{(k-1)})-\tilde\G_{\theta}(T_n,T_{n-1},u_{n-1}^{(k-1)})\nonumber\\
=&:a_k+b_{k-1}-a_{k-1},
\end{align}
where
\begin{align*}
a_{k}=& S(\delta T)u_{n-1}^{(k)}+\bi(1-\theta)\delta T S(\delta T)F(u_{n-1}^{(k)})
+\bi\theta\delta T F(a_k)+S(\delta T)Q^\frac12\delta_nW,\\
b_{k-1}=&S(\delta T)u_{n-1}^{(k-1)}+\bi\int_{T_{n-1}}^{T_n}S(T_n-s)F\left(u_{u_{n-1}^{(k-1)}}(s)\right)ds+\int_{T_{n-1}}^{T_n}S(T_n-s)Q^\frac12dW.
\end{align*}

Based on the Taylor expansion of $F(a_k)=F(a_{k-1})+F'(\tau_k)(a_k-a_{k-1})$ with $\tau_k$ being determined by $a_k$ and $a_{k-1}$, we derive
\begin{align*}
a_k-a_{k-1}=&S(\delta T)\left(u_{n-1}^{(k)}-u_{n-1}^{(k-1)}\right)+\bi(1-\theta)\delta T S(\delta T)\left(F(u_{n-1}^{(k)})-F(u_{n-1}^{(k-1)})\right)\\
&+\bi\theta\delta T \left(F(a_k)-F(a_{k-1})\right)\\
=&S(\delta T)\left(u_{n-1}^{(k)}-u_{n-1}^{(k-1)}\right)+\bi(1-\theta)\delta T S(\delta T)\left(F(u_{n-1}^{(k)})-F(u_{n-1}^{(k-1)})\right)\\
&+\bi\theta\delta T F'(\tau_k)(a_k-a_{k-1})
\end{align*}

Hence, scheme \eqref{solu3} can be expressed as
\begin{align*}
u_n^{(k)}=&S_{\theta,k}S(\delta T)u_{n-1}^{(k)}+\left(1-S_{\theta,k}\right)S(\delta T)u_{n-1}^{(k-1)}\\
&+\bi(1-\theta)\delta TS_{\theta,k} S(\delta T)\left(F(u_{n-1}^{(k)})-F(u_{n-1}^{(k-1)})\right)\\
&+\bi\int_{T_{n-1}}^{T_n}S(T_n-s)F\left(u_{u_{n-1}^{(k-1)}}(s)\right)ds+\int_{T_{n-1}}^{T_n}S(T_n-s)Q^\frac12dW,
\end{align*}
where $S_{\theta,k}:=(1-\bi\theta\delta T F'(\tau_k))^{-1}$.

\begin{tm}
Let Assumptions \ref{ap1} and \ref{ap2} hold with $s=0$, and $\{u_n^{(k)}\}_{0\le n\le N,k\in\N}$ be the solution of \eqref{solu3}. Then the proposed algorithm \eqref{solu3} converges to the exact solution as $k\to\infty$ over unbounded time domain if 
\begin{align*}
f(\theta):=\left(1+(2-\theta)L_F\delta T+L_F\delta Te^{L_F\delta T}\right)e^{-\alpha\delta T}<1.
\end{align*} 

Moreover, the accuracy of the convergence is faster than $[f(\theta)]^k$, which decreases as $\theta$ being larger.
\end{tm}

\begin{proof}
Based on the notation $\epsilon_n^{(k)}:=u_n^{(k)}-u(T_n)$ again, we derive
\begin{align*}
\epsilon_n^{(k)}=&S_{\theta,k}S(\delta T)\epsilon_{n-1}^{(k)}+\left(1-S_{\theta,k}\right)S(\delta T)\epsilon_{n-1}^{(k-1)}\\
&+\bi(1-\theta)\delta TS_{\theta,k} S(\delta T)\left(F(u_{n-1}^{(k)})-F(u_{n-1}^{(k-1)})\right)\\
&+\bi\int_{T_{n-1}}^{T_n}S(T_n-s)\left[F\left(u_{u_{n-1}^{(k-1)}}(s)\right)-F\left(u_{u(T_{n-1})}(s)\right)\right]ds.
\end{align*}
It then leads to 
\begin{align*}
&\|\epsilon_n^{(k)}\|_{L^2(\Omega;H)}\\
\le&\left(1+(1-\theta)L_F\delta T\right)\|S_{\theta,k}\|_{\mathcal{L}(H)}e^{-\alpha\delta T}\|\epsilon_{n-1}^{(k)}\|_{L^2(\Omega;H)}\\
&+\left(\|1-S_{\theta,k}\|_{\mathcal{L}(H)}+(1-\theta) L_F\delta T\|S_{\theta,k}\|_{\mathcal{L}(H)}\right)e^{-\alpha\delta T}\|\epsilon_{n-1}^{(k-1)}\|_{L^2(\Omega;H)}\\
&+ L_F\int_{T_{n-1}}^{T_n}e^{-\alpha(T_n-s)}\|G(s)\|_{L^2(\Omega;H)}ds
\end{align*}
with the notation $G(s):=u_{u_{n-1}^{(k-1)}}(s)-u_{u(T_{n-1})}(s)$. 
For operator $1-S_{\theta,k}$, we deduce
\begin{align*}
\|1-S_{\theta,k}\|_{\mathcal{L}(H)}
=\|S_{\theta,k}\|_{\mathcal{L}(H)}\|\bi\theta\delta TF'(\tau_k)\|_{\mathcal{L}(H)}
\le\theta L_F\delta T
\end{align*}
due to the fact $\|S_{\theta,k}\|_{\mathcal{L}(H)}<1.$

Moreover, according to the mild solution \eqref{mild}, we get for any $s\in[T_{n-1},T_{n}]$ that
\begin{align*}
\|G(s)\|_{L^2(\Omega;H)}=&\|u_{u_{n-1}^{(k-1)}}(s)-u_{u(T_{n-1})}(s)\|_{L^2(\Omega;H)}\\
\le&e^{-\alpha(s-T_{n-1})}\|\epsilon_{n-1}^{(k-1)}\|_{L^2(\Omega;H)}+ L_F\int_{T_{n-1}}^se^{-\alpha(s-r)}\|G(r)\|_{L^2(\Omega;H)}dr.
\end{align*}
Then the Gronwall inequality yields
\begin{align*}
\|G(s)\|_{L^2(\Omega;H)}\le& \left(1+L_F(s-T_{n-1})e^{L_F(s-T_{n-1})}\right)e^{-\alpha(s-T_{n-1})}\|\epsilon_{n-1}^{(k-1)}\|_{L^2(\Omega;H)}.
\end{align*}

Above estimations finally lead to
\begin{align*}
\|\epsilon_n^{(k)}\|_{L^2(\Omega;H)}
\le& \left(1+(1-\theta)L_F\delta T\right)e^{-\alpha\delta T}\|\epsilon_{n-1}^{(k)}\|_{L^2(\Omega;H)}\\
&+L_F\delta T\left(1+e^{L_F\delta T}\right)e^{-\alpha\delta T}\|\epsilon_{n-1}^{(k-1)}\|_{L^2(\Omega;H)}\\
=&:\gamma_1\|\epsilon_{n-1}^{(k)}\|_{L^2(\Omega;H)}+\gamma_2\|\epsilon_{n-1}^{(k-1)}\|_{L^2(\Omega;H)},
\end{align*}
where we have used the following estimation
\begin{align*}
&L_F\int_{T_{n-1}}^{T_n}\left(1+L_F(s-T_{n-1})e^{ L_F(s-T_{n-1})}\right)ds\\
=&L_F\delta Te^{L_F\delta T}+L_F\delta T+1-e^{L_F\delta T}
\le L_F\delta Te^{L_F\delta T}.
\end{align*}

Based on the arguments in Theorem \ref{tm3.2}, 
the error converge to zero as $k\to\infty$ if 
\begin{align*}
f(\theta)=\gamma_1+\gamma_2=\left(1+(2-\theta)L_F\delta T+L_F\delta Te^{L_F\delta T}\right)e^{-\alpha\delta T}<1.
\end{align*}
The convergence rate turns to be 
\begin{align*}
\|\varepsilon^{(k)}\|_{\infty}\le\left(\frac{\gamma_2}{1-\gamma_1}\right)^k\|\varepsilon^{(0)}\|_{\infty}=\left(\frac{f(\theta)-\gamma_1}{1-\gamma_1}\right)^k\|\varepsilon^{(0)}\|_{\infty}
<\left[f(\theta)\right]^k\|\varepsilon^{(0)}\|_{\infty}
\end{align*}
with $\varepsilon^{(k)}:=\left(\|\epsilon_1^{(k)}\|_{L^2(\Omega;H)},\cdots,\|\epsilon_n^{(k)}\|_{L^2(\Omega;H)}\right)^\top$. 

\begin{figure}[h]
\centering
\subfigure{
\begin{minipage}[t]{0.45\linewidth}
  \includegraphics[height=4.5cm,width=5.5cm]{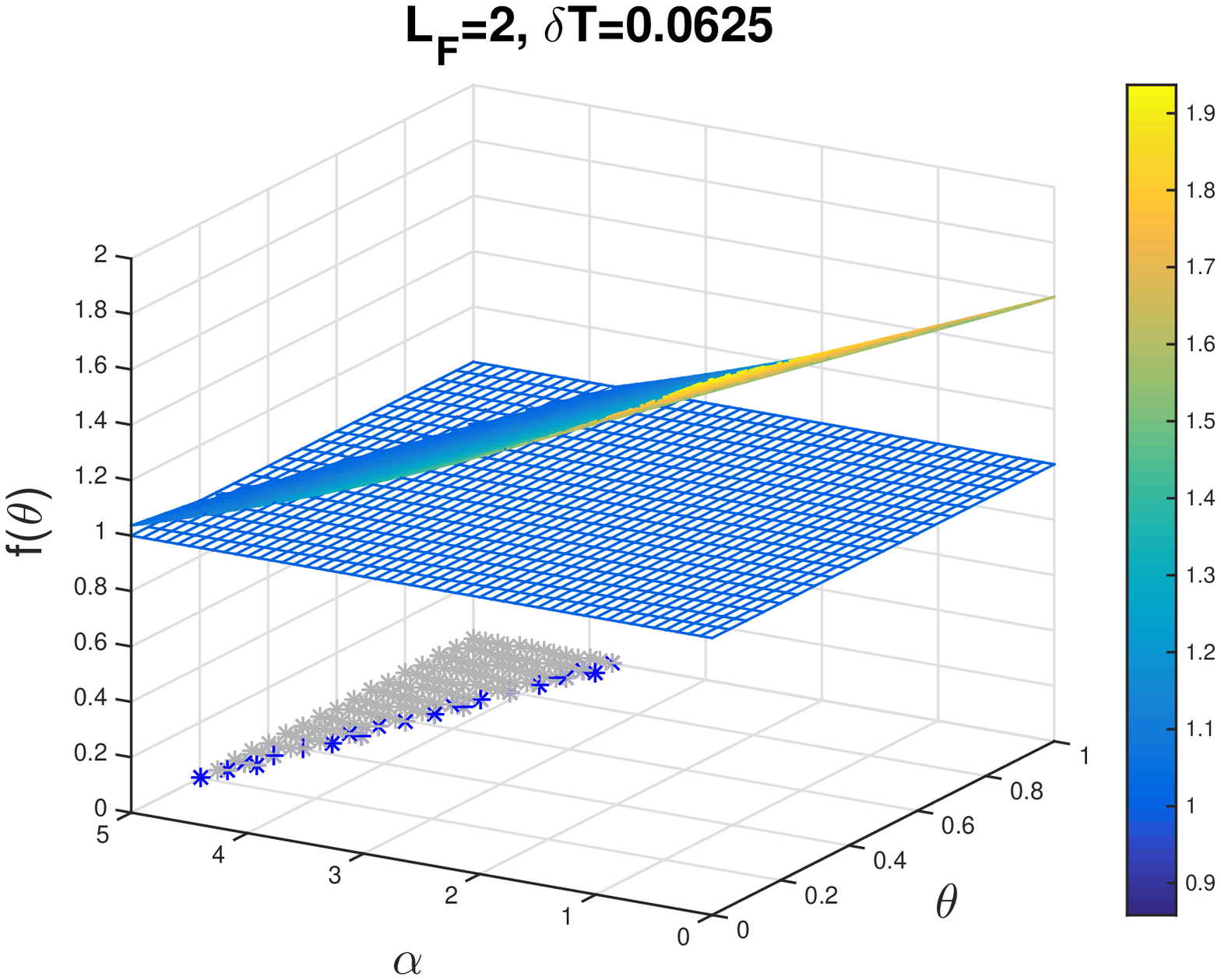}
  \end{minipage}
  }
  \subfigure{
  \begin{minipage}[t]{0.45\linewidth}
  \includegraphics[height=4.5cm,width=5.5cm]{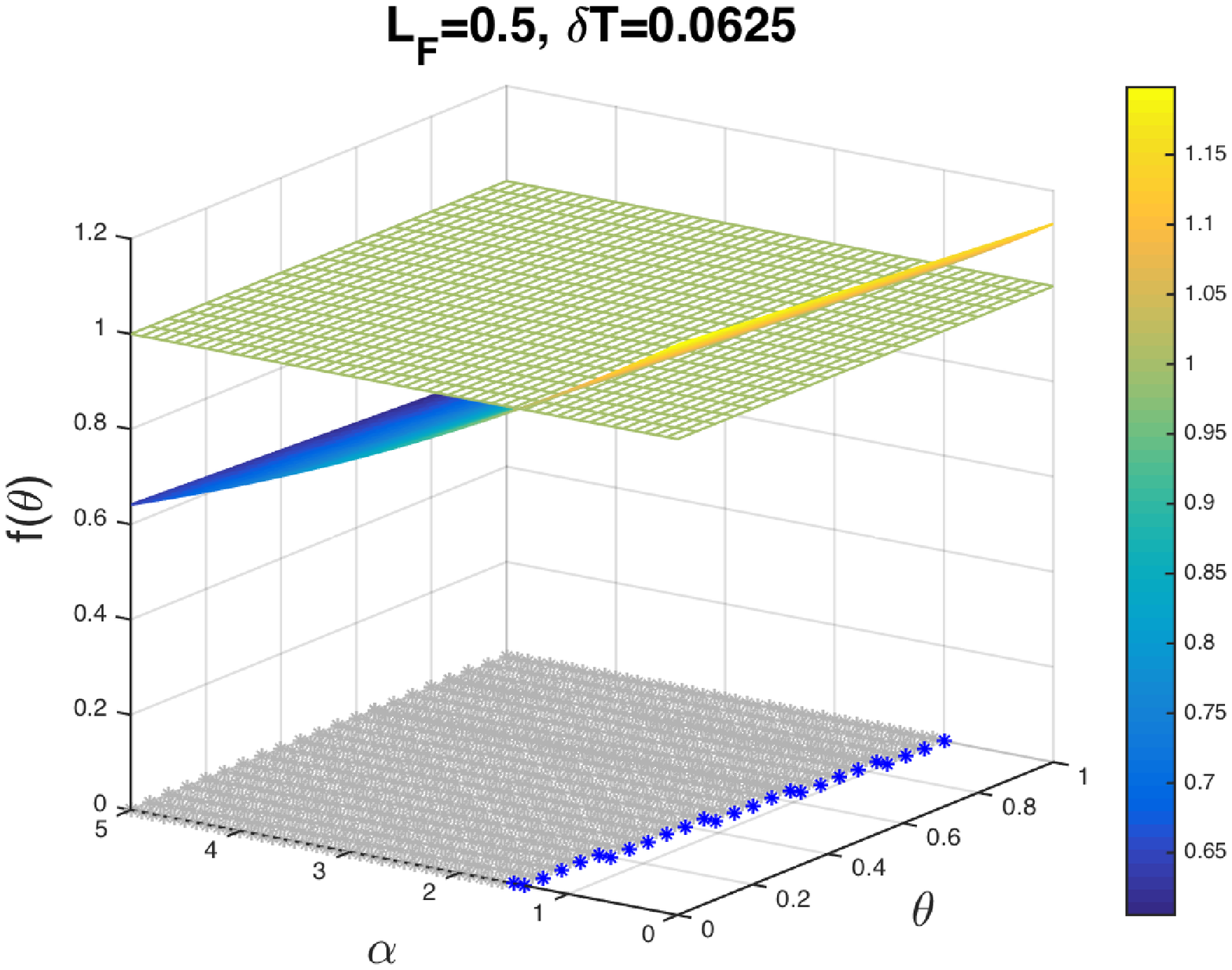}
  \end{minipage}
  }
  \caption{Convergence area (grey) of $f(\theta)$ vs. $\alpha$ and $\theta$.}
  \label{figstableNL}
\end{figure}
In addition, the fact $f'(\theta)=-L_F\delta Te^{-\alpha\delta T}<0$ indicates that the parareal exponential $\theta$-scheme converges faster when $\theta$ is larger, see Figure \ref{figstableNL}.
\end{proof}

\section{Numerical experiments}
\label{sec5}

This section is devoted to investigate the relationship between the convergence error and several parameters, i.e., $\alpha$, $\lambda$ and $\theta$, based on which we can find a proper number $k$ as the terminate iteration number for different cases.

We consider the linear equation \eqref{linearmodel} with initial value $u_0=0$.
Throughout the numerical experiments, we use the average of 1000 sample paths as an approximation of the expectation, and choose dimension $M=10$ for the spectral Galerkin approximation in spatial direction.

\begin{figure}[h]
\centering
\subfigure{
\begin{minipage}[t]{0.45\linewidth}
  \includegraphics[height=4.5cm,width=5.5cm]{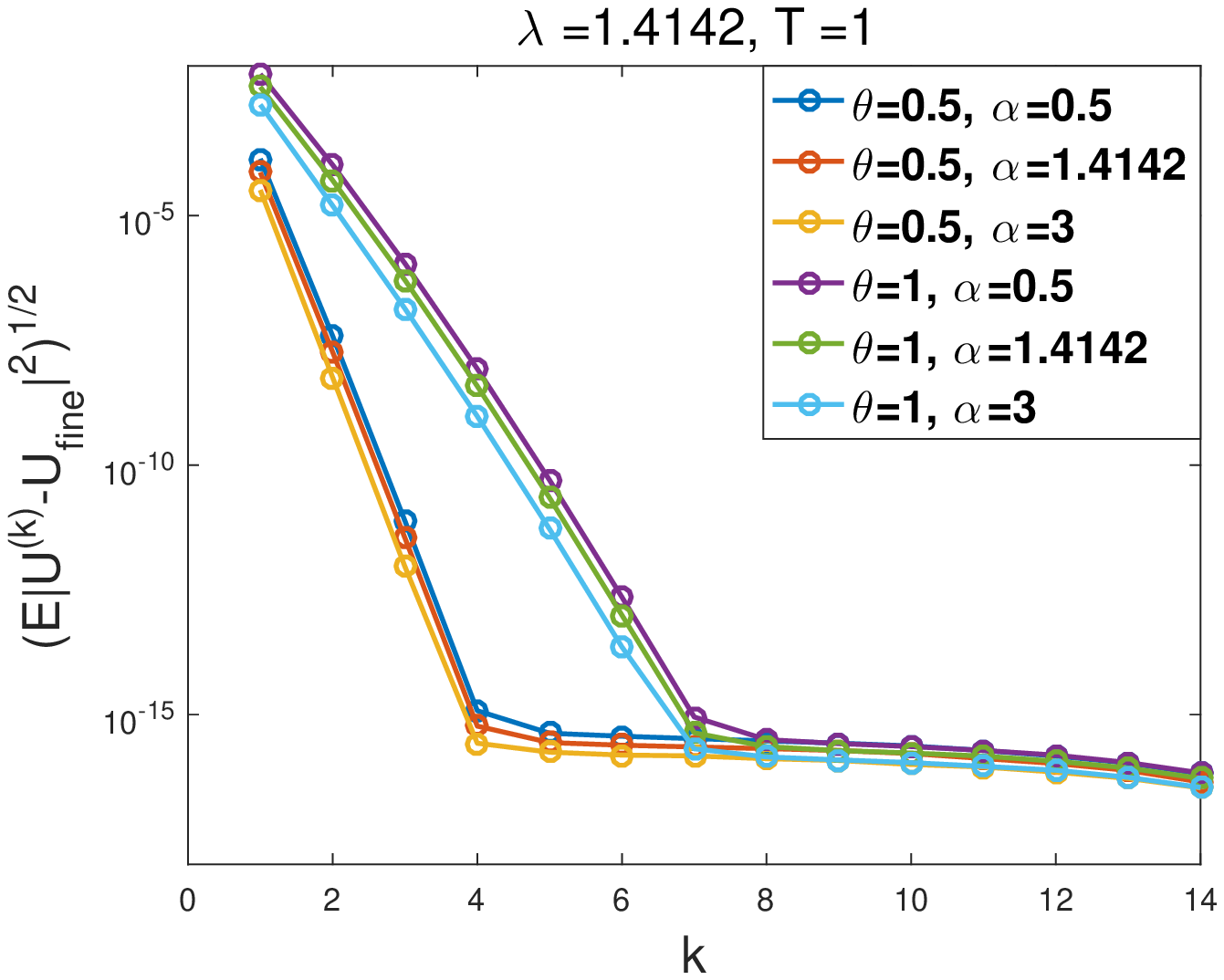}
  \end{minipage}
  }
  \subfigure{
  \begin{minipage}[t]{0.45\linewidth}
  \includegraphics[height=4.5cm,width=5.5cm]{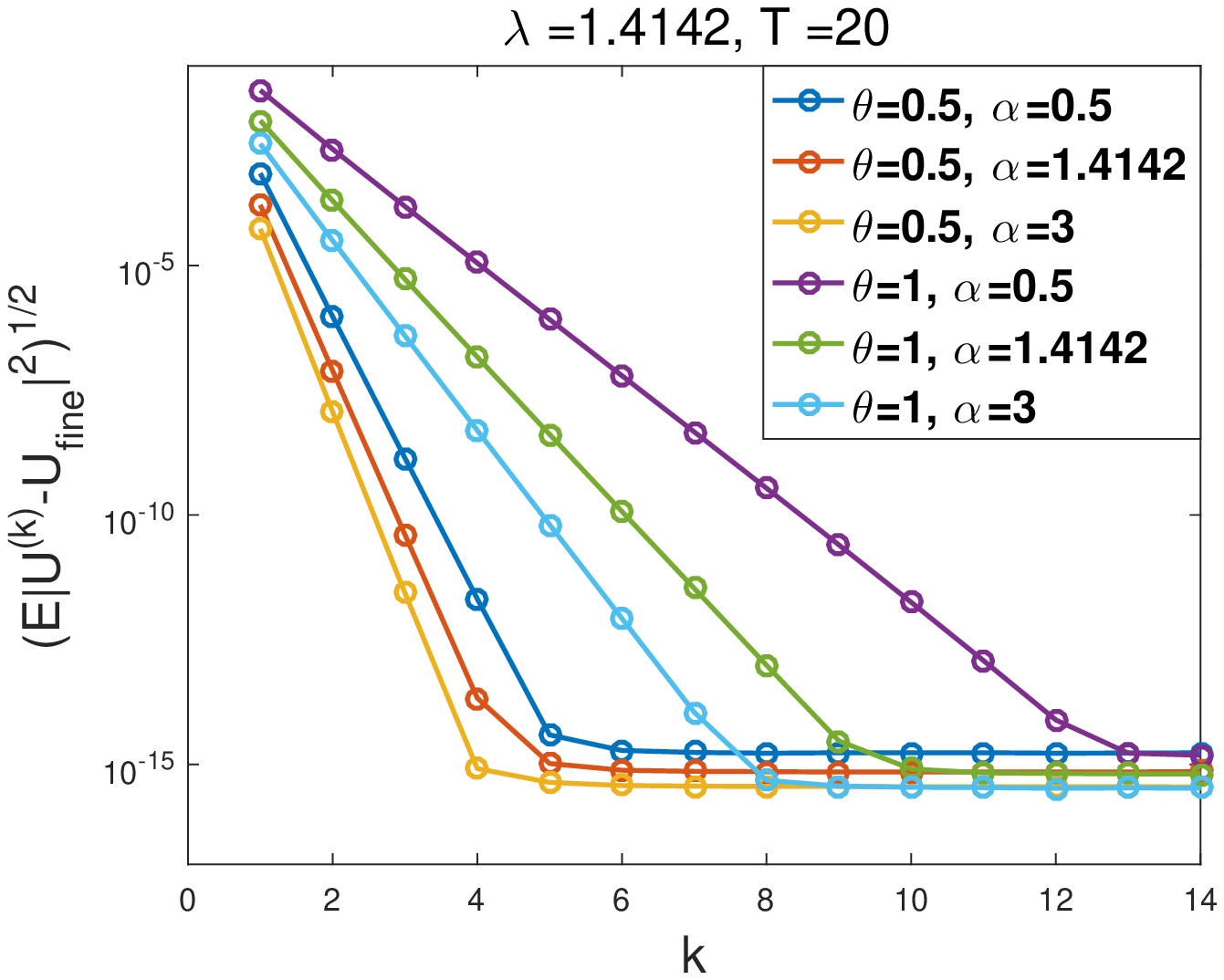}
  \end{minipage}
  }
  \caption{Mean square error $(\sup_{1\le n\le N}\E\|u_n^{(k)}-v_n\|^2)^\frac12$ vs. iteration number $k$ ($\lambda=\sqrt{2},\delta t=2^{-6}, J=4$).}
  \label{converge1}
\end{figure}

We get from Theorem \ref{tm3.2} that the time-uniform convergence holds for all $\lambda\in\R$ and $\alpha>0$ if $\theta\in[\frac12,1]$, which is illustrated in Figure \ref{converge1} for $\theta=0.5,1$ and time interval $T=1,20$. Figure \ref{converge1} shows the evolution of the mean square error $(\sup_{1\le n\le N}\E\|u_n^{(k)}-v_n\|^2)^\frac12$ with iteration number $k$. For $T=1$, the iteration number can be chosen as $k=4$ for $\theta=\frac12$ and $k=7$ when $\theta=1$, which coincides with the result that the convergence order is $2k$ instead of $k$ when $\theta=\frac12$. For larger time $T=20$, since the constant $C$ in Theorem \ref{tm3.2} is negatively correlated with $\alpha$ for $\theta\in[\frac12,1]$, the proposed algorithm also converges but with different iteration number $k$. 

\begin{figure}[h]
\centering
\subfigure{
\begin{minipage}[t]{0.45\linewidth}
  \includegraphics[height=4.5cm,width=5.5cm]{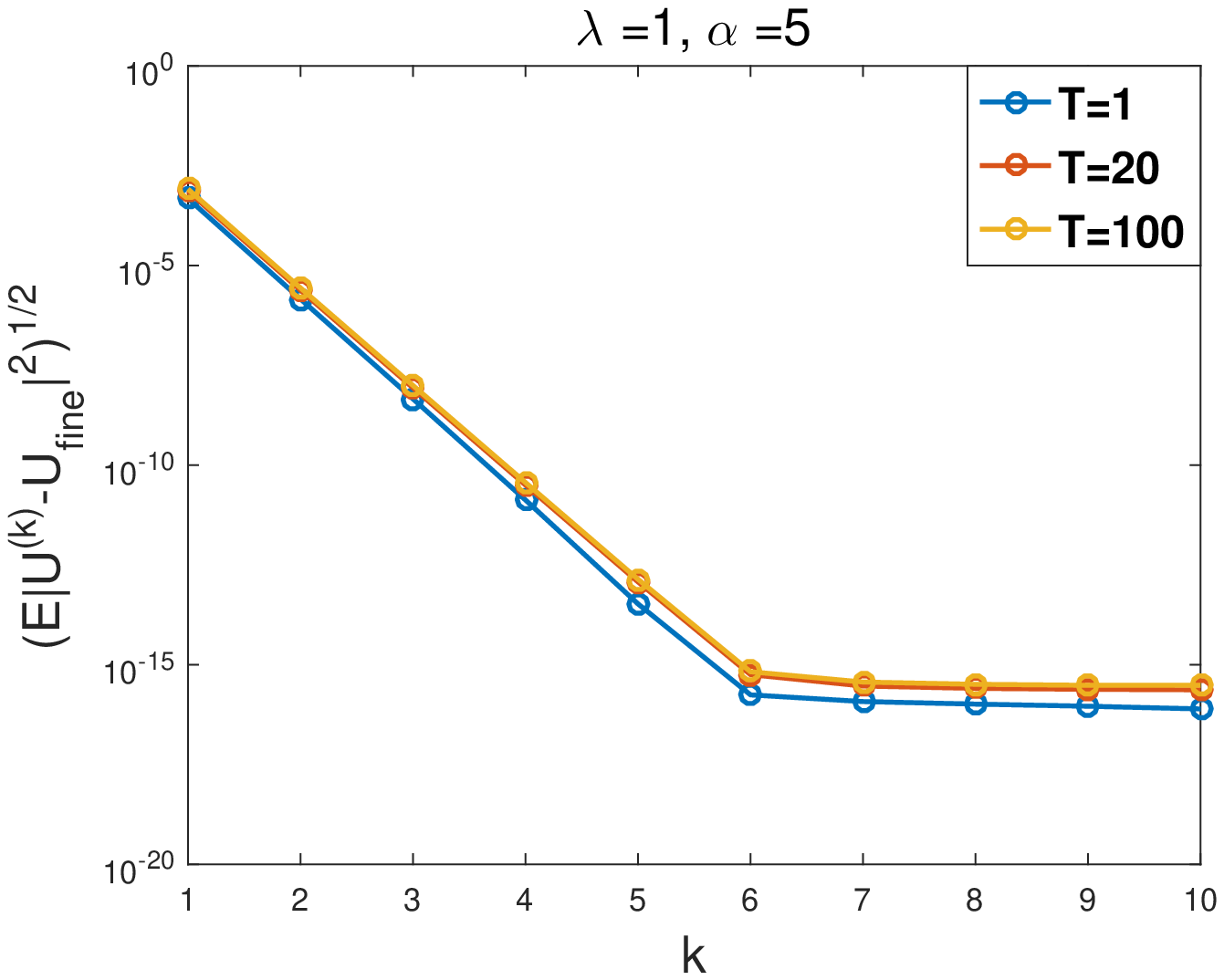}
  \end{minipage}
  }
  \subfigure{
  \begin{minipage}[t]{0.45\linewidth}
  \includegraphics[height=4.5cm,width=5.5cm]{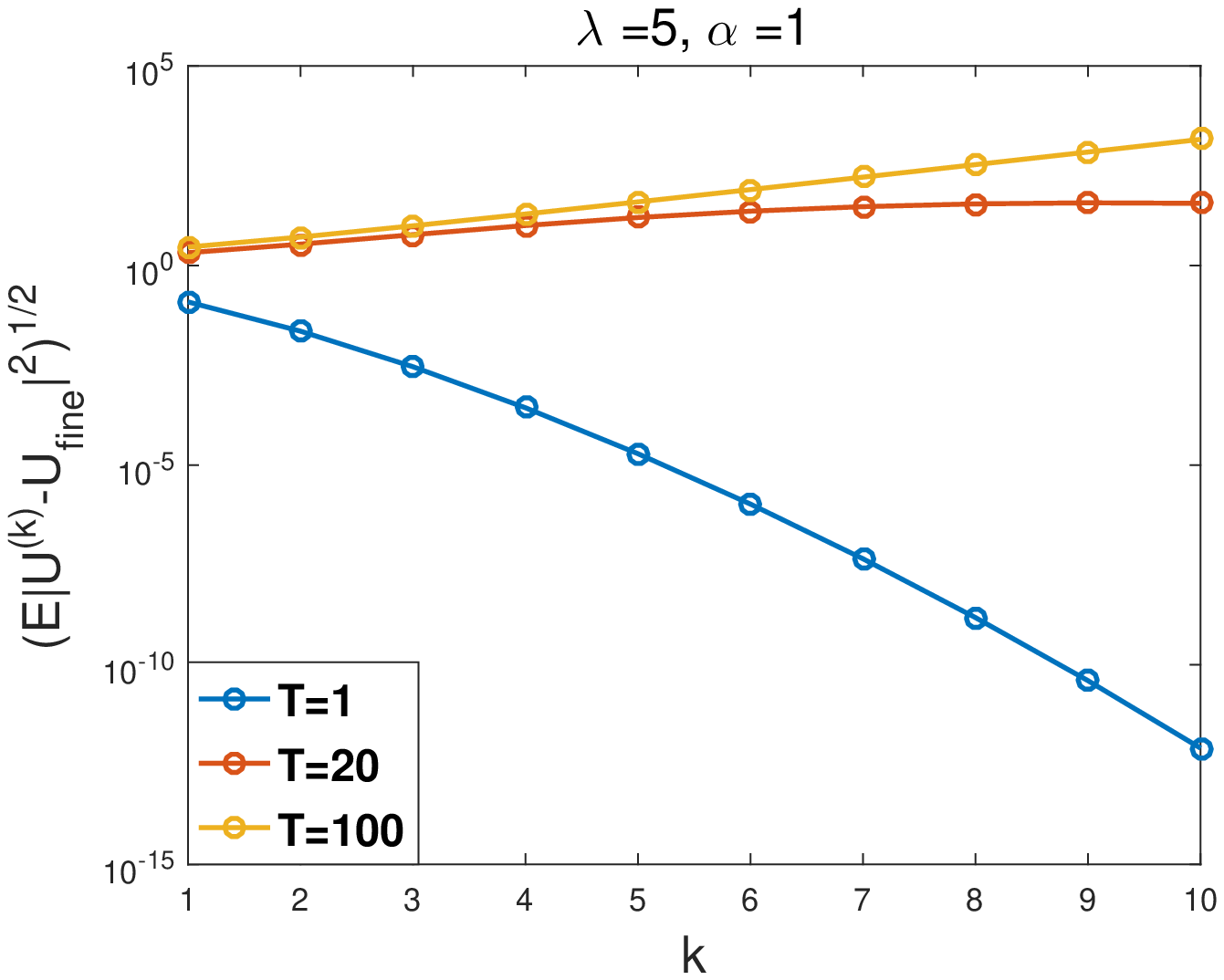}
  \end{minipage}
  }
  \caption{Mean square error $(\sup_{1\le n\le N}\E\|u_n^{(k)}-v_n\|^2)^\frac12$ vs. iteration number $k$.}
  \label{converge2}
\end{figure}

When $\theta\in[0,\frac12)$, the convergence result holds uniformly if $\alpha>\sqrt{\left(\frac12-\theta\right)}|\lambda|$ as stated in Theorem \ref{tm3.2}. Figure \ref{converge2} also shows evolution of the mean square error with respect to $k$ for $\theta=0$ and $T=1,20,100$. It can be find that if the condition $\alpha>\sqrt{\left(\frac12-\theta\right)}|\lambda|$ is not satisfied, e.g., $\lambda=5$, $\alpha=1$, the proposed algorithm diverges as time going larger. 

In particular, based on numerical experiments above, we now fix $k=3$ to verify the convergence order of the proposed scheme for different $\theta\in[0,1]$. 
Figure \ref{figorder} considers the convergence order of the proposed parareal algorithm for different $\lambda$ and $\alpha$ with fine step size $\delta t=2^{-8}$. The order turns to be $k$ for $\theta=0,0.4,0.55,0.9$, but increases to $2k$ when $\theta=\frac12$, which coincides with the result in Theorem \ref{tm3.2}.

\begin{figure}[h]
\centering
\subfigure{
\begin{minipage}[t]{0.45\linewidth}
  \includegraphics[height=4.5cm,width=5.5cm]{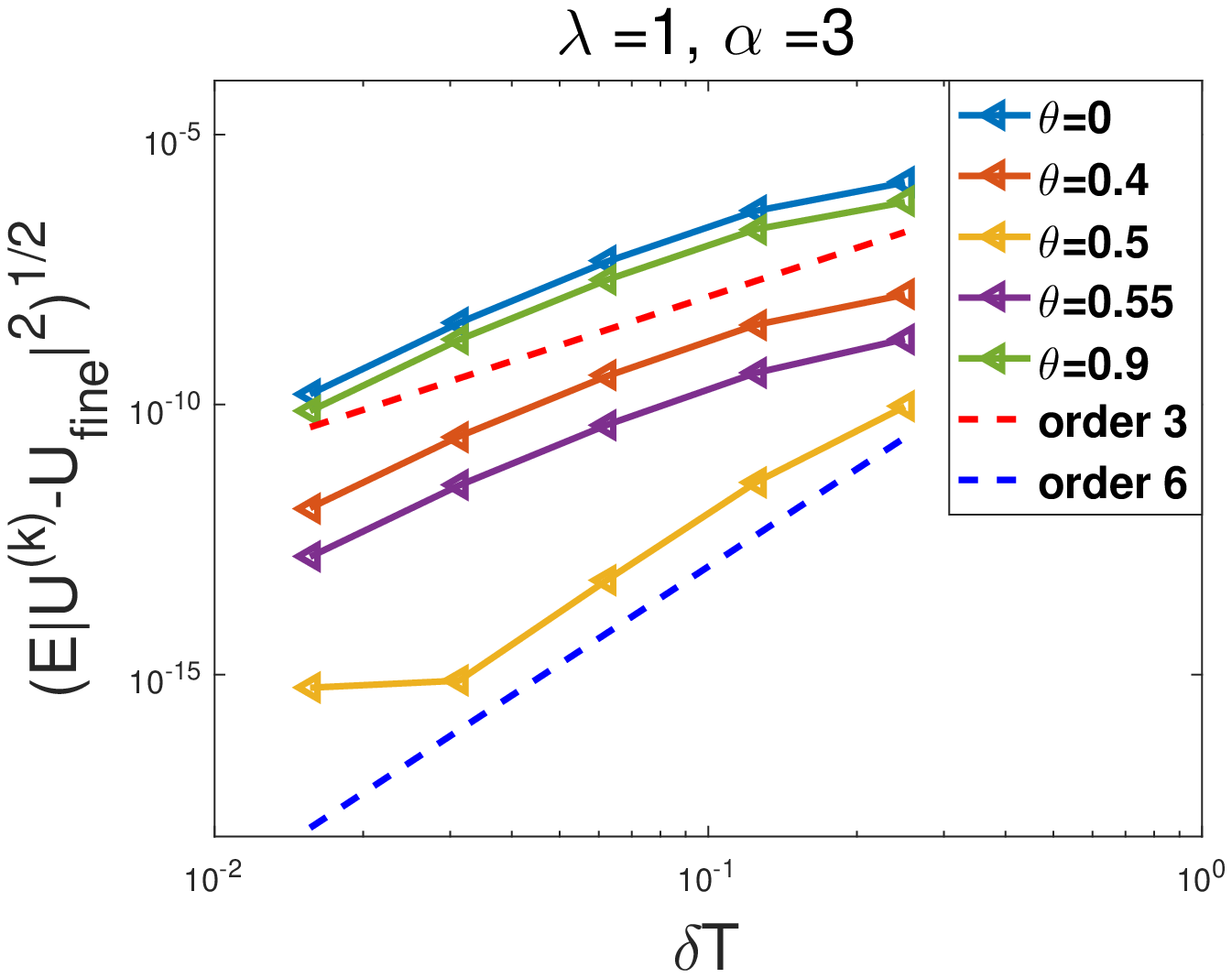}
  \end{minipage}
  }
  \subfigure{
  \begin{minipage}[t]{0.45\linewidth}
  \includegraphics[height=4.5cm,width=5.5cm]{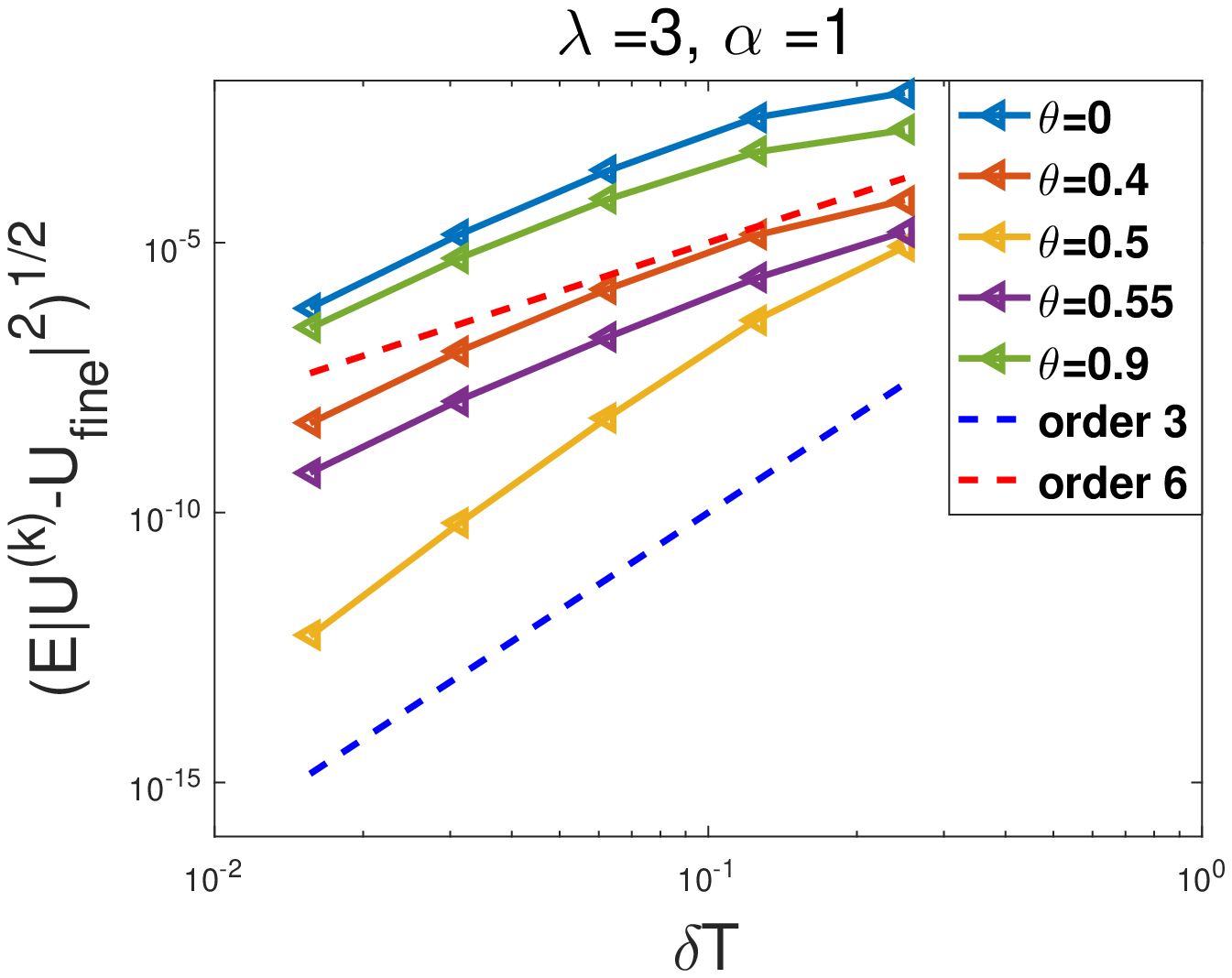}
  \end{minipage}
  }
  \caption{Mean square order with respect to $\delta T=2^{-i}$, $i=2,\cdots,6$.}
  \label{figorder}
\end{figure}

\section*{Appendix}
\subsection*{Proof of Theorem \ref{tm4.1}}
Since $\F$ is the exact propagator, it has the following expression 
\begin{align*}
\F(T_{n+1},T_{n},u_n^{(k-1)})=&S(\delta T)u_n^{(k-1)}+\bi\int_{T_n}^{T_{n+1}}S(T_{n+1}-s)F(u_{u_n^{(k-1)}}(s))ds\\
&+\int_{T_n}^{T_{n+1}}S(T_{n+1}-s)Q^\frac12dW(s),
\end{align*}
where $u_{u_n^{(k-1)}}(s)$ denotes the exact solution at time $s$ starting from $u_n^{(k-1)}$ at $T_n$.
Then algorithm \eqref{scheme} yields
\begin{align}\label{scheme2}
u_{n+1}^{(k)}=&S(\delta T)u_n^{(k)}+\bi  S(\delta T)F(u_n^{(k)})\delta T-\bi  S(\delta T)F(u_n^{(k-1)})\delta T\nonumber\\
&+\bi \int_{T_n}^{T_{n+1}}S(T_{n+1}-s)F(u_{u_n^{(k-1)}}(s))ds
+\int_{T_n}^{T_{n+1}}S(T_{n+1}-s)Q^\frac12dW(s),
\end{align}
compared with the exact solution
\begin{align*}
u(T_{n+1})=\F(T_{n+1},T_n,u(T_n)).
\end{align*}

Denoting the error $\epsilon_n^{(k)}:=u(T_n)-u_n^{(k)}$, we get
\begin{align*}
\epsilon_{n+1}^{(k)}=&S(\delta T)\epsilon_n^{(k)}
-\bi  S(\delta T)F(u_n^{(k)})\delta T
+\bi  S(\delta T)F(u_n^{(k-1)})\delta T\\
&+\bi \int_{T_n}^{T_{n+1}}S(T_{n+1}-s)F(u_{u(T_n)}(s))ds\\
&-\bi \int_{T_n}^{T_{n+1}}S(T_{n+1}-s)F(u_{u_n^{(k-1)}}(s))ds\\
=&S(\delta T)\epsilon_n^{(k)}
+\bi  S(\delta T)\left[F(u(T_n))-F(u_n^{(k)})\right]\delta T\\
&-\bi  S(\delta T)\left[F(u(T_n))-F(u_n^{(k-1)})\right]\delta T\\
&+\bi \int_{T_n}^{T_{n+1}}S(T_{n+1}-s)\left[F(u_{u(T_n)}(s))-F(u_{u_n^{(k-1)}}(s))\right]ds\\
=&:I+II-III+IV.
\end{align*}
Thus, the mean square error reads
\begin{align*}
\|\epsilon_{n+1}^{(k)}\|_{L^2(\Omega;H)}
\le\|I\|_{L^2(\Omega;H)}+\|II\|_{L^2(\Omega;H)}+\|III\|_{L^2(\Omega;H)}+\|IV\|_{L^2(\Omega;H)},
\end{align*}
where 
\begin{align}\label{I}
&\|I\|_{L^2(\Omega;H)}\le e^{-\alpha\delta T}\|\epsilon_n^{(k)}\|_{L^2(\Omega;H)},\\
&\|II\|_{L^2(\Omega;H)}\le   L_F\delta Te^{-\alpha\delta T}\|\epsilon_n^{(k)}\|_{L^2(\Omega;H)}
\end{align}
and
\begin{align}
\|III\|_{L^2(\Omega;H)}\le   L_F\delta Te^{-\alpha\delta T}\|\epsilon_n^{(k-1)}\|_{L^2(\Omega;H)}.
\end{align}

It then suffices to estimate term $IV$.
In fact, denoting $G(s):=u_{u(T_n)}(s)-u_{u_n^{(k-1)}}(s)$ and according to the mild solution \eqref{mild}, we obtain for any $s\in[T_n,T_{n+1}]$ that
\begin{align*}
\|G(s)\|_{L^2(\Omega;H)}=&\|u_{u(T_n)}(s)-u_{u_n^{(k-1)}}(s)\|_{L^2(\Omega;H)}\\
\le&e^{-\alpha(s-T_n)}\|\epsilon_n^{(k-1)}\|_{L^2(\Omega;H)}\\
&+ \left\|\int_{T_n}^sS(s-r)\left[F(u_{u(T_n)}(r))-F(u_{u_n^{(k-1)}}(r))\right]dr\right\|_{L^2(\Omega;H)}\\
\le&e^{-\alpha(s-T_n)}\|\epsilon_n^{(k-1)}\|_{L^2(\Omega;H)}+  L_F\int_{T_n}^se^{-\alpha(s-r)}\|G(r)\|_{L^2(\Omega;H)}dr.
\end{align*}
Then the Gronwall inequality yields
\begin{align*}
\|G(s)\|_{L^2(\Omega;H)}\le& \left(1+  L_F(s-T_n)e^{  L_F(s-T_n)}\right)e^{-\alpha(s-T_n)}\|\epsilon_n^{(k-1)}\|_{L^2(\Omega;H)}\\
\le&\left(1+  L_F\delta Te^{  L_F \delta T}\right)e^{-\alpha(s-T_n)}\|\epsilon_n^{(k-1)}\|_{L^2(\Omega;H)}.
\end{align*}
%with constant $C>0$ independent of $s$. 
As a result,
\begin{align}\label{IV}
\|IV\|_{L^2(\Omega;H)}\le&  L_F\int_{T_n}^{T_{n+1}}e^{-\alpha(T_{n+1}-s)}\|G(s)\|_{L^2(\Omega;H)}ds\nonumber\\
\le&\left(1+  L_F\delta Te^{  L_F \delta T}\right)  L_F\delta Te^{-\alpha\delta T}\|\epsilon_n^{(k-1)}\|_{L^2(\Omega;H)}.
\end{align}
Based on estimations \eqref{I}--\eqref{IV} and the fact that $\epsilon_0^{(k)}=0$ for all $k\in\N$, we derive
for $n=1,\cdots,N-1$ that
\begin{align}\label{epsilonk2}
\|\epsilon_{n+1}^{(k)}\|_{L^2(\Omega;H)}
\le&(1+  L_F\delta T)e^{-\alpha\delta T}\|\epsilon_n^{(k)}\|_{L^2(\Omega;H)}\nonumber\\
&+\left(2+  L_F\delta Te^{  L_F \delta T}\right)  L_F\delta Te^{-\alpha\delta T}\|\epsilon_n^{(k-1)}\|_{L^2(\Omega;H)}\nonumber\\
\le&\left(2+  L_F\delta Te^{  L_F \delta T}\right)  L_F\delta Te^{-\alpha\delta T}\sum_{j=1}^n\left(\beta^{n-j}\|\epsilon_j^{(k-1)}\|_{L^2(\Omega;H)}\right)
\end{align}
with the notation $\beta:=(1+  L_F\delta T)e^{-\alpha\delta T}>0$.
Denoting the error vector 
\begin{align*}
\varepsilon^{(k)}:=\left(\|\epsilon_1^{(k)}\|_{L^2(\Omega;H)},\cdots,\|\epsilon_N^{(k)}\|_{L^2(\Omega;H)}\right)^\top
\end{align*} 
and the $N$-dimensional matrix (see also \cite{gander2007})
\begin{equation*}M(\beta)=
\left(
\begin{array}{cccccc}
0&0&\cdots &0& 0\\
1&0&\cdots & 0&0\\
\beta&1&\cdots &0&0\\
\beta^{2}&\beta&\cdots&0&0\\
\vdots & \vdots& \vdots &\vdots&\vdots \\
\beta^{N-2}&\beta^{N-3}&\cdots&1&0
\end{array}
\right),
\end{equation*}
we can rewrite \eqref{epsilonk2} as
\begin{align*}
\varepsilon^{(k)}\le C\delta Te^{-\alpha\delta T}M(\beta)\varepsilon^{(k-1)}
\le (C\delta Te^{-\alpha\delta T})^kM^k(\beta)\varepsilon^{(0)}.
\end{align*}

It is shown in \cite{gander2007} that  
\begin{align*}
\|M^k(\beta)\|_\infty\le\frac{(N-1)(N-2)\cdots(N-k)}{k!}(\beta\vee1)^{N-k-1}
\le\frac{N^k}{k!}\left(\beta^{N}\vee1\right),
\end{align*}
which leads to the first result in the theorem:
\begin{align*}
\|\varepsilon^{(k)}\|_\infty
\le& \left(C\delta Te^{-\alpha\delta T}\right)^k\frac{N^k}{k!}\left(\beta^{N}\vee1\right)\|\varepsilon^{(0)}\|_\infty\\
\le&\left(e^{-\alpha\delta T}\right)^k\frac{(CT)^k}{k!}\left(e^{(  L_F-\alpha)T}\vee1\right)\|\varepsilon^{(0)}\|_\infty.
\end{align*}

Note that the function $f(\delta T):=e^{-\alpha\delta T}-\delta T$ is continuous and takes value in $(e^{-\alpha}-1,1]$ for $\delta T\in[0,1)$. Hence, there exists some $\delta T_*=\delta T_*(\alpha)\in(0,1)$ such that $f(\delta T)\le0$ for any $\delta T\in[\delta T_*,1)$.
In fact, $\delta T_*$ satisfies that $\delta T_*^{-1}\ln\delta T_*^{-1}=\alpha$, which decreases when $\alpha$ increases.

\subsection*{Proof of Theorem \ref{tm4.2}}
Note that 
\begin{align}\label{vn0}
v_{n,0}=&v_{n-1,J}=\F_{I}(t_{n-1,J},t_{n-1,J-1},v_{n-1,J-1})\nonumber\\
=&S(\delta T)v_{n-1,0}+\bi \sum_{l=1}^JS(l\delta t)F(v_{n-1,J-l})\delta t+\sum_{l=1}^JS(l\delta t)Q^\frac12\delta_{n,J+1-l}W.
\end{align}
Similarly, we get
\begin{align}\label{hatu}
\hat u_{n-1,J}^{(k-1)}=&\F_{I}(t_{n-1,J},t_{n-1,J-1},\hat u_{n-1,J-1}^{(k-1)})\nonumber\\
=&S(\delta T)u_{n-1}^{(k-1)}+\bi \sum_{l=1}^JS(l\delta t)F(\hat u_{n-1,J-l}^{(k-1)})\delta t+\sum_{l=1}^JS(l\delta t)Q^\frac12\delta_{n,J+1-l}W.
\end{align}

In the following, we still denote the above error by $\epsilon_n^{(k)}:=u_n^{(k)}-v_{n,0}$ for convenience, which has the same symbol as in the proof of Theorem \ref{tm4.1} but with different meaning. 
Then we can decompose the error into several parts
\begin{align*}
\epsilon_{n}^{(k)}
=&\left(\G_{I}(T_{n},T_{n-1},u_{n-1}^{(k)})-v_{n,0}\right)
-\left(\G_{I}(T_{n},T_{n-1},u_{n-1}^{(k-1)})-v_{n,0}\right)+\hat u_{n-1,J}^{(k-1)}-v_{n,0}\\
=&S(\delta T)\epsilon_{n-1}^{(k)}+\bi  \left(S(\delta T)F(u_{n-1}^{(k)})\delta T
-S(\delta T)F(v_{n-1,0})\delta T\right)\\
&-\bi \left(S(\delta T)F(u_{n-1}^{(k-1)})\delta T
-S(\delta T)F(v_{n-1,0})\delta T\right)\\
&+\bi \left(\sum_{l=1}^JS(l\delta t)F(\hat u_{n-1,J-l}^{(k-1)})\delta t-\sum_{l=1}^JS(l\delta t)F(v_{n-1,J-l})\delta t\right)\\
=&:\tilde{I}+\tilde{II}-\tilde{III}+\tilde{IV}
\end{align*}
according to \eqref{vn0} and \eqref{hatu}.
For the first three terms, we derive
\begin{align*}
\|\tilde{I}\|_{L^2(\Omega;H)}\le& e^{-\alpha\delta T}\|\epsilon_{n-1}^{(k)}\|_{L^2(\Omega;H)},\\
\|\tilde{II}\|_{L^2(\Omega;H)}\le&   L_F\delta T e^{-\alpha\delta T}\|\epsilon_{n-1}^{(k)}\|_{L^2(\Omega;H)}
\end{align*}
and
\begin{align*}
\|\tilde{III}\|_{L^2(\Omega;H)}\le C\delta T e^{-\alpha\delta T}\|\epsilon_{n-1}^{(k-1)}\|_{L^2(\Omega;H)}.
\end{align*}

To get the estimation of term $\tilde{IV}$, we define 
$\tilde{G}_{j}:=\hat u_{n-1,j}^{(k-1)}-v_{n-1,j}$ for any $j=0,\cdots,J$, 
then \eqref{vn0} and \eqref{hatu} yields
\begin{align*}
\|\tilde G_{j}\|^2_{L^2(\Omega;H)}=&\Big\|S(\delta T)u_{n-1}^{(k-1)}+\bi \sum_{l=1}^{j}S(l\delta t)F(\hat u_{n-1,j-l}^{(k-1)})\delta t\\
&-\Big(S(\delta T)v_{n-1,0}+\bi \sum_{l=1}^{j}S(l\delta t)F(v_{n-1,j-l})\delta t\Big)\Big\|^2_{L^2(\Omega;H)}\\
\le&2e^{-2\alpha\delta T}\|\epsilon_{n-1}^{(k-1)}\|^2_{L^2(\Omega;H)}
+2j\delta t^2L_F^2\sum_{l=1}^je^{-2\alpha l\delta t}\|\tilde G_{j-l}\|^2_{L^2(\Omega;H)}\\
\le&2e^{-2\alpha\delta T}\|\epsilon_{n-1}^{(k-1)}\|^2_{L^2(\Omega;H)}
+2\delta T\delta tL_F^2\sum_{m=0}^{j-1}e^{-2\alpha (j-m)\delta t}\|\tilde G_{m}\|^2_{L^2(\Omega;H)}.
\end{align*}
Equivalently, it can be written as 
\begin{align*}
&e^{2\alpha j\delta t}\|\tilde G_{j}\|^2_{L^2(\Omega;H)}\\
\le&2e^{-2\alpha(\delta T-j\delta t)}\|\epsilon_{n-1}^{(k-1)}\|^2_{L^2(\Omega;H)}
+2\delta T\delta tL_F^2\sum_{m=0}^{j-1}e^{2\alpha m\delta t}\|\tilde G_{m}\|^2_{L^2(\Omega;H)}.
\end{align*}
According to the discrete Gronwall inequality, we get
\begin{align*}
&\|\tilde G_{j}\|^2_{L^2(\Omega;H)}\\
\le&2e^{-2\alpha\delta T}\|\epsilon_{n-1}^{(k-1)}\|^2_{L^2(\Omega;H)}\Big(1
+e^{-2\alpha j\delta t}\sum_{0\le m<j}e^{2\alpha m\delta t}2\delta T\delta tL_F^2(1+2\delta T\delta tL_F^2)^{j-m-2}\Big)\\
\le&Ce^{-2\alpha\delta T}\|\epsilon_{n-1}^{(k-1)}\|^2_{L^2(\Omega;H)}
\end{align*}
with $C$ independent of $j$. Hence,
\begin{align*}
\|\tilde{IV}\|^2_{L^2(\Omega;H)}
\le&\delta T\delta tL_F^2\sum_{l=1}^Je^{-2\alpha l\delta t}\|\tilde G_{J-l}\|^2_{L^2(\Omega;H)}\\
\le&C\delta T^2e^{-2\alpha\delta T}\|\epsilon_{n-1}^{(k-1)}\|^2_{L^2(\Omega;H)}.
\end{align*}

In conclusion, we get
\begin{align*}
\|\epsilon_{n}^{(k)}\|^2_{L^2(\Omega;H)}
\le(1+  L_F\delta T)e^{-\alpha\delta T}\|\epsilon_{n-1}^{(k)}\|^2_{L^2(\Omega;H)}+C\delta T e^{-\alpha\delta T}\|\epsilon_{n-1}^{(k-1)}\|^2_{L^2(\Omega;H)},
\end{align*}
which leads to the final results based on the procedure in the proof of Theorem \ref{tm4.1}.

\bibliography{parareal}
\bibliographystyle{plain}

\end{document}